\documentclass[letterpaper]{siamart251104}
\usepackage[
  letterpaper,
  textwidth=5.125in,
  textheight=598pt,
  centering,
  includeheadfoot,
  voffset=0pt
]{geometry}

\usepackage{amssymb}
\usepackage{amsmath}
\usepackage{bbm}
\usepackage{mathtools}
\usepackage[svgnames]{xcolor}
\usepackage{graphicx}
\usepackage{adjustbox}
\usepackage{placeins}
\usepackage{lipsum}
\usepackage[english]{babel}
\usepackage[utf8]{inputenc}
\usepackage{cancel}
\usepackage{subcaption}
\usepackage{tikz}
\usepackage{tikz-cd}
\usetikzlibrary{arrows.meta, positioning}
\graphicspath{{./}}
\usepackage{algpseudocode}

\usetikzlibrary{calc}
\definecolor{ink}{HTML}{253C4A}
\definecolor{cone}{HTML}{6B95B0}
\definecolor{section}{HTML}{E6A342}
\definecolor{sectionedge}{HTML}{A46616}
\definecolor{axis}{HTML}{72828C}

\newsiamremark{example}{Example}

\DeclareMathOperator*{\argmax}{arg\,max}
\DeclareMathOperator*{\argmin}{arg\,min}

\DeclareMathOperator{\sign}{sgn}
\DeclareMathOperator{\diag}{diag}
\DeclareMathOperator{\conv}{Conv}

\DeclareMathOperator{\Stab}{Stab}

\DeclareMathOperator{\Tr}{Tr}

\newcommand{\lr}[2]{\left\langle #1 , #2 \right\rangle}
\newcommand{\p}[1]{\mathbb{P} \big ( #1 \big )}

\newcommand{\T}{\mathsf{T}}
\newcommand{\1}{1_n}

\title{Certifiably Correct Fr\'echet Means in Orbit Spaces}

\author{
  Dustin G. Mixon\thanks{Department of Mathematics, The Ohio State University}
  \and
  Tom Needham\thanks{Department of Mathematics, Florida State University}
  \and
  Nathan Willey\footnotemark[1]
}

\headers{Certifiably Correct Fréchet Means}{D.G. Mixon, T. Needham, and N. Willey}

\begin{document}

\maketitle

\begin{abstract}
We study Fr\'echet means in orbit spaces of finite orthogonal group actions. 
Since computing these means is $\mathsf{NP}$-hard in general, we take inspiration from group synchronization to develop tractable spectral and semidefinite relaxations. 
Pairwise alignment certifies global optimality and we develop a semidefinite certificate that can extend beyond this condition.
For Gaussian data modulo sign, we prove that our certificate succeeds in a high-noise regime where alignment fails. 
Numerical experiments evaluate these methods in various settings, including summarizing ensembles of districting plans and clustering time series.
\end{abstract}

\section{Introduction}
\label{sec:intro}

Given a collection of data points $x_1,\ldots,x_n$ in some space~$\mathcal{X}$, what point $x\in\mathcal{X}$ does a good job of representing the data points?
In settings where the space is Euclidean, one is inclined to use the mean $\frac{1}{n}\sum_{i=1}^nx_i$, but there are some settings in which the mean fails to capture important features of the data.
For example, if each $x_i$ is a black-and-white image of someone's face (vectorized so that each coordinate represents a pixel value), then the mean will exhibit some undesirable ghostly features.
The issue here is that natural images form a low-dimensional and nonconvex subset of Euclidean space, and thanks to the lack of convexity, sample means fail to look natural.
Is there a suitable alternative to the mean in settings where the space is nonconvex?

Recall that, given points $x_1,\ldots,x_n$ in a Euclidean space $V$, the Euclidean mean is the point $\mu\in V$ that minimizes the mean squared error:
\[
\frac{1}{n}\sum_{i=1}^n\|x_i-\mu\|^2.
\]
We can take inspiration from this to treat more general settings in which our space~$\mathcal{X}$ has a notion of distance $d_{\mathcal{X}}$.
Define the \textbf{Fr\'{e}chet functional} 
$F\colon\mathcal{X}\to\mathbb{R}$ by
\[
    F(m)
=\frac{1}{n}\sum_{i=1}^n d_{\mathcal{X}}(m,x_i)^2,
\]
and say $\mu\in\mathcal{X}$ is a \textbf{Fr\'{e}chet mean} of $x_1,\ldots,x_n$ if $\mu$ globally minimizes $F$.
Unfortunately, thanks to the nonlinear geometry of $\mathcal{X}$, there is typically not enough convexity in the Fr\'{e}chet functional for global minimizers to be easily obtained. 
As such, one is forced to make do with local minimizers, which are known as \textbf{Karcher means}.

The notion of a Fr\'echet mean originates in the work of Fr\'echet~\cite{frechet1948elements} and was later developed separately in the Riemannian setting by Karcher and Kendall~\cite{karcher1977riemannian, kendallProbabilityConvexityHarmonic1990}. 
These ideas were later adapted to statistical shape analysis, where means on nonlinear and nonconvex shape spaces are a central object of study~\cite{kendall1989survey,le2000frechet}.
In modern data science, Fr\'echet means arise as Wasserstein barycenters in optimal transport~\cite{zemel2019frechet,needhamGeometricAveragesPartitioned2021}, in signal processing~\cite{kurtek2011signal}, and in topological data analysis~\cite{turner2014frechet}. 
The Fr\'echet mean is useful in practice because it provides a representative point for non-Euclidean data without requiring vector space structure.

In this paper, we are interested in Fr\'{e}chet means in a certain ubiquitous setting in data science.
Suppose our data points are represented by vectors, but with the implicit understanding that there was a non-canonical choice made in the vector representation.
For example, a simple unlabeled $n$-vertex graph with edge weights might be represented by a weighted adjacency matrix, but only after deciding how to index the vertices.
Alternatively, given a biological specimen, a scientist might record the $(x,y,z)$-coordinates of $n$ specified landmarks, but only after deciding on a reference frame.
In each of these cases, the choice of representation can be expressed as a group of ambiguities:
in the first case, the group $S_n$ acts on the adjacency matrix by conjugation, while in the second case, the group $SO(3)$ acts on the left of the $3\times n$ matrix whose columns give the origin-centered landmark coordinates.

In general, the data we are interested in lies in a finite-dimensional Hilbert space $\left (V, \langle \cdot , \cdot \rangle \right)$ with norm $\| \cdot \|$ associated with its inner product. 
In addition, there is a known orthogonal group action $G \curvearrowright V$ that captures the underlying ambiguity in the vector representation of our data. 
As such, we think of our data as residing in the orbit space $\mathcal{X}=V/G$ consisting of $G$-orbits 
\[[x] \coloneqq G \cdot x \text{ for } x \in V.\]
By the orthogonality of the group action, the orbit space can be equipped with the natural quotient metric 
\begin{equation}
    \label{eq:dQ}
    d([x], [y]) \coloneqq \inf_{g \in G} \|x-gy\|.
\end{equation}
Unless stated otherwise, we take $G$ to be finite, in which case the above infimum will be achieved by at least one group element.

\begin{example}
\label{ex:pm}
We will consider \textit{sign invariance} as a running example throughout this work. 
Here, we take $V = \mathbb{R}^d$ and $G = \{\pm I_d\}$, which we denote by $G=\pm$.
The metric~\eqref{eq:dQ} is then given by $d([x], [y]) = \min \left\{ \|x-y\|, \|x+y\|\right\}$, and the quotient space $\mathbb{R}^d/\pm$ can be identified with a cone over $\mathbb{R} \mathbb{P}^{d-1}$. Figure~\ref{fig:pm_cone} shows the quotient identification and resulting geometry in the special case where $d=2$. 
\end{example}
\begin{figure}
    \centering
    \includegraphics[width=0.4\linewidth,trim={0 0 500 0},clip]{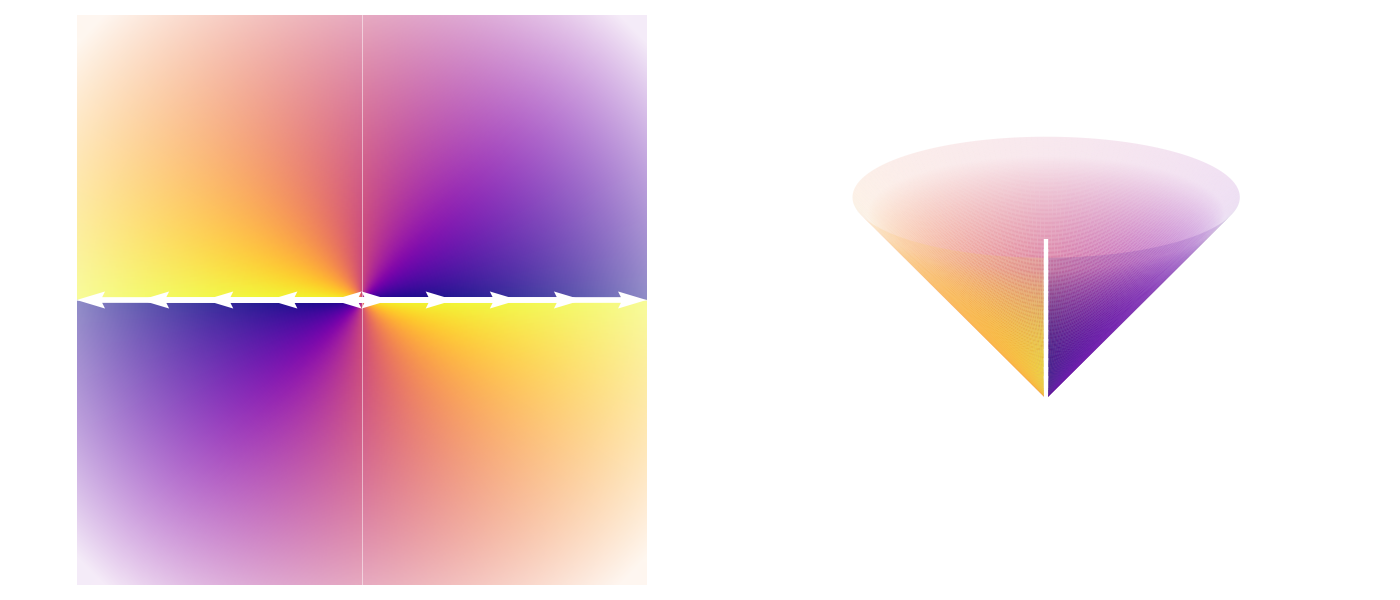}
    \qquad\includegraphics[width=0.4\linewidth,trim={600 100 100 100},clip]{Z2_quotient_space.png}
    \caption{Illustration of $\mathbb{R}^2/\pm$. Left: The plane with the identification $x \sim -x$ shown in color. Right: The quotient space can be visualized as a cone after identifying each direction with its negation, e.g., by discarding the lower halfplane and then gluing the positive horizontal axis to the negative horizontal axis.}
    \label{fig:pm_cone}
\end{figure}

As illustrated in Figure~\ref{fig:pm_cone}, orbit spaces exhibit a level of nonconvexity that insists upon a theory of Fr\'{e}chet means.
In this setting, it will be convenient to consider the Fr\'{e}chet functional as the $G$-invariant functional
\begin{equation}
    \label{eqn:frechet}
    F(m) =\frac{1}{n}\sum_{i=1}^n \min_{g \in G} \| m - gx_i \|^2
\end{equation}
over $V$, rather than a functional over the quotient $V/G$.

In the next section, we identify basic properties of Fr\'echet and Karcher means in orbit spaces, including generic uniqueness of the Fr\'echet mean and $\mathsf{NP}$-hardness of its computation. 
Section~\ref{sec:alignment} then introduces the notion of \textit{alignment} and shows that the Euclidean mean of aligned orbit representatives represents a Fr\'echet mean. 
Next, Section~\ref{sec:relaxations} develops spectral and semidefinite relaxations inspired by group synchronization and derives computable certificates of global optimality. 
In the sign-invariant setting under a Gaussian model, Section~\ref{sec.case study} establishes that these certificates succeed well beyond the regime of pairwise alignment. 
Section~\ref{sec:numerics} numerically evaluates the methods in a variety of settings, and then we conclude with a discussion in Section~\ref{sec:discussion}.
\section{The state of play}
\label{sec:state}

In this section, we report a few basic results about Fr\'echet and Karcher means in orbit spaces: Karcher means are Euclidean means, Fr\'echet means are generically unique, Fr\'echet means are hard to find in the worst case, and traditional approaches to find such means are insufficient. %

\subsection{Karcher means are Euclidean means}
\label{sec:karcher_are_euc}

We start with a structure theorem for Karcher means in orbit spaces.

\begin{theorem}
    \label{thm:euc_mean}
    Let $[y]$ be a Karcher mean of $\{[x_i]\}_{i=1}^n$.
    Then $y$ is the Euclidean mean of orbit representatives $\{g_i x_i \}_{i=1}^n$ for some $g_1,\ldots,g_n\in G$.
\end{theorem}

To prove this, we will use what is presumably a standard fact from convex analysis:

\begin{lemma}
\label{lem.local min of min of cvx}
Given convex functions $f_1,\ldots,f_n\colon V\to\mathbb{R}$, suppose $x_0\in V$ is a local minimizer of $f= \min_i f_i$.
Then $x_0$ is also a global minimizer of $f_{i^\star}$ for each $i^\star\in\arg\min_if_i(x_0)$.
\end{lemma}

\begin{proof}
We prove the contrapositive.
Suppose $x_0\in V$ is not a global minimizer of $f_{i^\star}$ for some $i^\star\in\arg\min_if_i(x_0)$.
Select any $x_1$ satisfying $f_{i^\star}(x_1) < f_{i^\star}(x_0)$ and for each $t\in(0,1)$, put $x_t\coloneqq(1-t)x_0+tx_1$.
Then for every $t\in(0,1)$, we have
\[
\min_i f_i(x_t)
\leq f_{i^\star}(x_t)
\leq(1-t)f_{i^\star}(x_0)+tf_{i^\star}(x_1)
<f_{i^\star}(x_0)
=f(x_0).
\]
Taking $t$ arbitrarily small, it follows that $x_0$ is not a local minimizer of $f$.
\end{proof}

\begin{proof}[Proof of Theorem~\ref{thm:euc_mean}]
The Fr\'echet functional \eqref{eqn:frechet} can be written as
\[
F(m)
=\min_{g_1,\ldots,g_n\in G}\frac{1}{n}\sum_{i=1}^n\|m-g_ix_i\|^2.
\]
In particular, $F$ is the minimum of finitely many convex functions.
Lemma~\ref{lem.local min of min of cvx} then gives that any local minimizer of $F$ must be a global minimizer of some
\[
m\longmapsto\frac{1}{n}\sum_{i=1}^n\|m-g_ix_i\|^2,
\]
i.e., the Euclidean mean of $\{g_ix_i\}_{i=1}^n$.
\end{proof}

Theorem~\ref{thm:euc_mean} suggests an alternating algorithm to find Fr\'echet means: alternate between (1) finding closest orbit representatives $\{g_ix_i\}_{i=1}^n$ to $m$ and (2) computing the Euclidean mean of $\{g_ix_i\}_{i=1}^n$.
This amounts to running the $k$-means iteration with $k=|G|$ on $\{gx_i\}_{i=1,\,g\in G}^n$ after initializing with a $G$-orbit of proto-centroids.
This approach, sometimes referred to as the ``max-max'' algorithm, is investigated in~\cite{devilliersInconsistencyTemplateEstimation2017}.
Unfortunately, such alternating algorithms are known to get stuck in local minimizers, and accordingly, we do not analyze this iteration in this paper.

\subsection{The Fr\'{e}chet mean is generically unique}

Next, we show that the Fr\'echet mean is unique for generic data:

\begin{theorem}
\label{thm:unique}
Fix any finite $G \leq O(V)$.
For generic $\{x_i\}_{i=1}^n \subset V$,  the Fr\'echet mean of $\{[x_i]\}_{i=1}^n \subset V/G$ is unique.
\end{theorem}

The generic condition is meant in the semi-algebraic sense and, as such, the datasets $\{x_i\}_{i=1}^n$ with multiple Fr\'echet means form a Lebesgue null set in $V^n$. 
Still, as the following example illustrates, uniqueness need not hold for arbitrary data.

\begin{example}
Let $u,v\in\mathbb{R}^d$ be nonzero orthogonal vectors, and consider their sign-invariant orbits $[u], [v] \in \mathbb{R}^d/\pm$. 
This two-point dataset has two different Fr\'echet means in orbit space: $[\frac{u+v}{2}]$ and $[\frac{u - v}{2}]$. 
Meanwhile, any perturbation that makes the vectors non-orthogonal results in a dataset with a unique Fr\'echet mean.
\end{example}

To prove Theorem~\ref{thm:unique}, we will use the following lemma, originally observed in~\cite{devilliersInconsistencyTemplateEstimation2017}. We provide a proof here for completeness.

\begin{lemma}
    \label{lemma:equal_norm}
    Fix any finite $G \leq O(V)$ and $\{x_i\}_{i=1}^n \subset V$.
    Then any two Fr\'echet means of $\{[x_i]\}_{i=1}^n\subset V/G$ have the same norm.
    \label{lemma:means}
\end{lemma}
\begin{proof}
    Given $m\in V$, write $m = \|m\| \widehat{m}$, where $\|\widehat{m}\| = 1$. Then the Fr\'echet functional~\eqref{eqn:frechet} can be written as
    \begin{equation}
        \label{eqn:polar_frechet}
        F(m) = \|m\|^2 - \frac{2}{n} \|m\| \sum_{i=1}^n \max_{g \in G} \langle  \widehat{m}, gx_i \rangle + \frac{1}{n} \sum_{i=1}^n \|x_i\|^2.
    \end{equation}
    We decompose the minimization over $m$ as
    \[
        \min_{m \in V} F(m) = \min_{\widehat{m} \in S(V)} \min_{\|m\|} F(\|m\|\widehat{m}).
    \]
    For fixed $\widehat{m}$, $F$ is quadratic in $\|m\|$ and so is minimized at
    \[
        \|m\| = \bigg( \frac{1}{n} \sum_{i=1}^n \max_{g \in G} \langle \widehat{m}, gx_i \rangle \bigg)_+,
    \]
    where $\left(x\right)_+ \coloneqq \max\{x,0\}$.
    Substituting back into \eqref{eqn:polar_frechet}, the minimization reduces to
    \[
         \min_{m \in V} F(m) = -\max_{\widehat{m} \in S(V)} \bigg( \frac{1}{n} \sum_{i=1}^n \max_{g \in G} \langle \widehat{m}, gx_i \rangle \bigg)_+^2 + \frac{1}{n} \sum_{i=1}^n \|x_i\|^2. \
    \]
    Therefore the minimum of $F$ determines $( \frac{1}{n} \sum_{i=1}^n \max_{g \in G} \langle \widehat{m}, gx_i \rangle )_+^2$, which in turn determines the norm $\| m\|$.
\end{proof}

\begin{proof}[Proof of Theorem~\ref{thm:unique}]
Denote the Fr\'echet functional for $\{x_i\}_{i=1}^n$ at $m$ by
\[
F(m;\{x_i\})
\coloneqq\frac{1}{n}\sum_{i=1}^n\min_{g\in G}\|m-gx_i\|^2.
\]
Letting $\mathbf{g} = (g_1, \dots, g_n) \in G^n$ denote a choice of group elements, 
we reduce the problem to minimizing over the finite set $G^n$:
\begin{align*}
        \min_{m \in V} F(m; \{x_i\}) 
        &=  \min_{m \in V} \min_{\mathbf{g} \in {G}^n} \frac{1}{n}\sum_{i=1}^n \| m - g_i x_i\|^2 \\
        &= \min_{\mathbf{g} \in {G}^n} \min_{m \in V} \frac{1}{n}\sum_{i=1}^n \| m - g_i x_i\|^2 \\
        &= \min_{\mathbf{g} \in {G}^n}  \frac{1}{n}\sum_{i=1}^n \bigg\|  \bigg( \frac{1}{n}\sum_{j=1}^n g_jx _j \bigg) - g_i x_i \bigg \|^2 \\
        &=: \min_{\mathbf{g} \in G^n} F(\mathbf{g}; \{x_i\}_{i=1}^n).
\end{align*}
The functional $\mathbf{g}\mapsto F(\mathbf{g}; \{x_i\}_{i=1}^n)$ is invariant under the diagonal action $h\cdot\mathbf g=(hg_1,\dots,hg_n)$, so without loss of generality we fix the representative by taking $g_1=I$ 
and restrict to $\{I\}\times G^{n-1}$.
Next consider the collection of data for which two distinct sets of group elements attain the Fr\'echet mean:
\begin{equation}
    \label{eqn:nullset}
    N= \bigcup_{\mathbf g\neq\mathbf h\in\{I\}\times G^{n-1}} \Big\{\{x_i\}_{i=1}^n: \ F(\mathbf g;\{x_i\})=F(\mathbf h;\{x_i\}) = \min_{\mathbf a \in G^n}F(\mathbf a;\{x_i\})\Big\}.
\end{equation}
For any distinct $\mathbf g, \mathbf h$ attaining the Fr\'echet mean,
\[
\Big\|\frac{1}{n}\sum_{j=1}^n g_j x_j \Big\|^2 - \Big\| \frac{1}{n}\sum_{j=1}^n h_j x_j \Big\|^2 = 0
\]
by Lemma~\ref{lemma:equal_norm}.
Expanding and subtracting the squared norms in $F$ then shows that
\[ 
\Big\{\{x_i\}_{i=1}^n: \ F(\mathbf g;\{x_i\})=F(\mathbf h;\{x_i\})\Big\}\subseteq \ker K_{\mathbf g,\mathbf h} ,
\]
where
\[
    K_{\mathbf g,\mathbf h}(\{x_i\}_{i=1}^n)\coloneqq\sum_{i,j}\langle x_j,(g_j^{-1}g_i-h_j^{-1}h_i)x_i\rangle.
\]
For a fixed $\mathbf g , \mathbf h$, $ K_{\mathbf g,\mathbf h}$ defines a  polynomial functional on \(V^n\). 
We now show that $ K_{\mathbf g,\mathbf h}$ is nontrivial. 
Since $\mathbf g\neq\mathbf h$ there is some index $j\ge2$ with $g_j\neq h_j$.
Then the \((1,j)\) and \((j,1)\) terms yield
\[
2\langle x_1,(g_j-h_j)x_j\rangle,
\]
which is a nonzero quadratic form, and so $x_1, x_j$ can be chosen to make it nonzero; letting $x_k=\mathbf 0$ for $k \notin \{1, j\}$ yields a nonzero point in the image of $K_{\mathbf g,\mathbf h}$. 
Hence $K_{\mathbf g,\mathbf h}$ is not identically zero, and as such, $\ker K_{\mathbf g,\mathbf h}$ is a proper algebraic subset of $V^n$. 
By \eqref{eqn:nullset}, $N$ is contained in a finite union of proper algebraic subsets. 
Thus, the Fr\'echet mean is generically unique.
\end{proof}

\subsection{Computing the Fr\'echet mean is $\mathsf{NP}$-hard}

Although the Fr\'echet mean is generically unique, it is in general $\mathsf{NP}$-hard to find, even in the special case of $\mathbb{R}^d/\pm$.

\begin{proposition}
\label{prop.np hard}
    Computing the Fr\'echet mean in $\mathbb{R}^d/\pm$ is $\mathsf{NP}$-hard when the dimension $d$ is part of the input. 
\end{proposition}

(Despite this result, it may be the case that computing the Fr\'echet mean in $\mathbb{R}^d/\pm$ is tractable for any fixed $d$.
We do not approach this problem here.)

\begin{proof}[Proof of Proposition~\ref{prop.np hard}]
We prove the result by reduction from $\operatorname{MaxCut}$, a well-known $\mathsf{NP}$-hard problem~\cite{garey1976simplified}.
Given a graph $H$ with $n$ vertices and $d$ edges, consider an oriented edge-vertex incidence matrix $X\in\mathbb{R}^{d\times n}$; to be explicit, if edge $i$ is incident to vertices $j$ and $k$ with $j<k$, then we may define $X$ by $X_{ij}=1$, $X_{ik}=-1$, and $X_{il}=0$ for all $l\not\in\{j,k\}$.
Then the Laplacian of $H$ is given by $L=X^\T X$. 
Recall that
$$
\operatorname{MaxCut}(H)
=\max_{\epsilon\in\{\pm1\}^n}\frac{1}{4}\epsilon^\T L\epsilon
=\frac{1}{4}\max_{\epsilon\in\{\pm1\}^n}\|X\epsilon\|^2.
$$
Let $x_i$ denote the $i$th column of $X$. 
Then the Fr\'echet functional associated with $x_1,\ldots,x_n$ satisfies
$$
F(m)
=\frac{1}{n}\sum_{i=1}^n\min\{\|m-x_i\|^2,\|m+x_i\|^2\}
=\min_{\epsilon\in\{\pm1\}^n}\frac{1}{n}\sum_{i=1}^n\|m-\epsilon_i x_i\|^2,
$$
and so
$$
\min_{m\in\mathbb{R}^d}F(m)
=\min_{m\in\mathbb{R}^d}\min_{\epsilon\in\{\pm1\}^n}\frac{1}{n}\sum_{i=1}^n\|m-\epsilon_i x_i\|^2
=\min_{\epsilon\in\{\pm1\}^n}\min_{m\in\mathbb{R}^d}\frac{1}{n}\sum_{i=1}^n\|m-\epsilon_i x_i\|^2.
$$
For a fixed $\epsilon$, the minimizer of $\frac{1}{n}\sum_{i=1}^n\|m-\epsilon_i x_i\|^2$ is given by the Euclidean mean $\frac{1}{n}X\epsilon$, in which case
\begin{align*}
\frac{1}{n}\sum_{i=1}^n\bigg\|\frac{1}{n}X\epsilon-\epsilon_i x_i\bigg\|^2
&=\bigg\|\frac{1}{n}X\epsilon\bigg\|^2
-2\bigg\langle \frac{1}{n}X\epsilon,\frac{1}{n}\sum_{i=1}^n\epsilon_i x_i\bigg\rangle
+\frac{1}{n}\sum_{i=1}^n\|\epsilon_i x_i\|^2\\
&=\frac{1}{n}\sum_{i=1}^n\|x_i\|^2-\frac{1}{n^2}\|X\epsilon\|^2.
\end{align*}
Overall,
$$
\min_{m\in\mathbb{R}^d}F(m)
=\min_{\epsilon\in\{\pm1\}^n}\bigg(\frac{1}{n}\sum_{i=1}^n\|x_i\|^2-\frac{1}{n^2}\|X\epsilon\|^2\bigg)
=\frac{1}{n}\sum_{i=1}^n\|x_i\|^2-\frac{1}{n^2}\max_{\epsilon\in\{\pm1\}^n}\|X\epsilon\|^2.
$$
It follows that
$$
\operatorname{MaxCut}(H)
=\frac{1}{4}\max_{\epsilon\in\{\pm1\}^n}\|X\epsilon\|^2
=\frac{n^2}{4}\bigg(\frac{1}{n}\sum_{i=1}^n\|x_i\|^2-\min_{m\in\mathbb{R}^d}F(m)\bigg).
$$
Thus, globally minimizing the Fr\'echet functional for $\mathbb{R}^d/\pm$ yields $\operatorname{MaxCut}(H)$.
\end{proof}

\subsection{The Riemannian approach to Fr\'echet means is limited}

Traditionally, Fr\'echet means are computed using tools from Riemannian geometry ~\cite{pennecIntrinsicStatisticsRiemannian2006, groisser2004newton}.
We can transfer these ideas to the orbit-space setting, but first, we need to wrestle with certain singularities.
For example, Figure~\ref{fig:pm_cone} illustrates how the origin of $\mathbb{R}^2$ becomes a cone-point singularity in the orbit space $\mathbb{R}^2/\pm$.
Here, the origin is special because it has a nontrivial stabilizer.
In general, we say $x\in V$ is an \textbf{irregular point} of the $G$-action if it is fixed by some non-identity element of $G$, and otherwise, it is \textbf{regular}.

Note that regularity is defined as a property of points in $V$ and not $V/G$. 
Regularity is still well defined on orbits; if $x \in V$ is irregular and fixed by some $g \in G$ then for $y = hx \in [x]$ we have that $(h g h^{-1})\, y = y $, and hence $y$ is also irregular.
In our setting, regularity is a generic condition on points in $V$ since the set of irregular points is exactly a finite union of fixed-point subspaces.

Irregular points mark a warping of the geometry of $V$ in the quotient metric. 
If a point $x \in V$ is irregular, then locally the space is not Euclidean and instead shares the geometry of $V/ \Stab(x)$ where $\Stab(x) \coloneqq \{ g \in G : gx = x\}$.
These irregularities give the orbit space $V/G$ a global orbifold structure rather than one of a traditional manifold. 
Despite this, after removing irregular points, $V/G$ is a flat Riemannian manifold whose induced Riemannian distance agrees locally with the quotient metric $d_{V/G}$.
In particular, the space can be locally identified with $V$.
This structure gives the possibility of finding the Fr\'echet mean by gradient descent. 
In fact, we can transfer a classical result from Kendall~\cite{kendallProbabilityConvexityHarmonic1990} to obtain a light guarantee on the nonexistence of spurious local minima of the Fr\'echet functional.

\begin{proposition}
    \label{prop:riemannian}
    Fix any finite nontrivial $G\leq O(V)$, a regular point $y \in V$, and orbits $\{[x_i]\}_{i=1}^n\subset V/G$ such that
    \[
    \max_{i\in\{1,\ldots,n\}}\min_{p\in[x_i]}\|p-y\|
    <r
    \coloneqq\frac{1}{2}\min_{\substack{g \in G \\ g \neq I}} \operatorname{dist}\big(y, (y - g y)^\perp\big),
    \]
    where $\operatorname{dist}(x,S)$ denotes the minimum Euclidean distance between the point~$x$ and the subspace~$S$.
    Then there is exactly one Karcher mean of $\{[x_i]\}_{i=1}^n$ contained in the image of the open ball $B(y,r) \subset V$ under the quotient map.
\end{proposition}

\begin{proof}
Since \(y\) is regular and any irregular point lies in some \((y-gy)^\perp\), the choice of \(r\) gives that \(B(y,r)\) contains no irregular points and
\[
gB(y,r)\cap B(y,r)=\emptyset
\qquad\text{for all }g\neq I.
\]
In particular, the quotient map \(\pi:B(y,r)\to\widetilde B\subset V/G\) is injective and \(\widetilde B\) inherits a flat Riemannian structure.

We next show that every pair of points in \(\widetilde B\) is joined by a unique minimizing geodesic that remains in the ball. Let \(p,q\in B(y,r)\). For every \(g\neq I\),
\begin{align*}
\|p-gq\|
&\geq \|y-gy\|-\|p-y\|-\|gq-gy\| \\
&= \|y-gy\|-\|p-y\|-\|q-y\| 
> 4r-2r
=2r,
\end{align*}
whereas $\|p-q\|\leq \|p-y\|+\|q-y\|<2r$.
Thus, $d([p],[q])=\|p-q\|$,
and this distance is uniquely attained by the identity. Taking \(p=y\) shows that \(\widetilde B\) is the geodesic ball of radius \(r\) about \([y]\). Hence the unique geodesic from \([p]\) to \([q]\) lifts to the line segment between \(p\) and \(q\), which remains in \(B(y,r)\).

By hypothesis, the support of the data lies within $\widetilde B$.
This is a regular geodesic ball: its sectional curvature is zero, and its radius is smaller than the injectivity radius at $[y]$. 
Thus Theorem~7.3 of~\cite{kendallProbabilityConvexityHarmonic1990} gives uniqueness of a Karcher mean within $\widetilde B$. 
Existence follows by taking the Euclidean mean of the representatives in $B(y,r)$.
\end{proof}

Here is one interpretation of this result: 
If we initialize gradient descent of the Fr\'echet functional at a regular point $y$ that satisfies the hypothesis of Proposition~\ref{prop:riemannian}, and if the iteration stays in the ball~$B(y,r)$, then it converges to the unique local minimizer in that ball.
Much like Theorem~\ref{thm:euc_mean}, the Karcher mean in the ball is represented by the Euclidean mean of the points \(p_i\in[x_i]\cap B(y,r)\). 
By requiring $r$ to be small, we ensure that the quotient map preserves pairwise distances.
In the next section, we formalize this \emph{alignment} property, identify a multitude of sets that satisfy this property, and prove that it guarantees global optimality of the Euclidean mean.
\section{Alignment}
\label{sec:alignment}

The previous section identified some algorithmic difficulties with computing Fr\'echet means.
In this section, we identify a family of problem instances for which computing the Fr\'echet mean is utterly painless.

\subsection{Introduction to aligned and alignable sets}
\label{sec.alignment}

We start with a definition:
\begin{definition}
    A set $S \subseteq V$ is \textbf{aligned} with respect to $G$ if for every $x,y \in S$,
    $$d([x],[y]) = \| x - y \|.$$
\end{definition}
Alternatively stated, $S$ is aligned when the quotient map $V \to V/G$ restricts to an isometric embedding of $S$.\footnote{This idea is discussed in~\cite{devilliersInconsistencyTemplateEstimation2017} using the nomenclature ``congruent sections."} 
As we will soon see, since aligned data exhibits identical geometry to the corresponding points in orbit space, when computing Fr\'echet means, one may equivalently work directly with the original data, which is easier thanks to its Euclidean setting.

For shorthand, we will say two points $x,y \in V$ are aligned if the set $\{x,y\}$ is aligned, and a sequence $\{x_i\}_{i=1}^n$ of points in $V$ is aligned if the underlying set $\{x_1,\ldots,x_n\}$ is aligned.
Two points are aligned precisely when the minimum in their quotient distance is attained by the identity group element. 
We say that a set is \textbf{maximally aligned} if it is aligned and not strictly contained in any other aligned set. Some examples follow below.

\begin{example}
\label{ex:perm_alignment}
\label{ex:reflection}
Let $G=S_d$ act on $V=\mathbb{R}^d$ by coordinate permutations.
Two points are aligned precisely when the vector entries are sorted in the same order.
More generally, when $G \leq O(V)$ is a reflection group, then the closure of any Weyl chamber is a maximally aligned set.
\end{example}

\begin{example}
    \label{ex:rotation}
    Suppose $g \in SO(2)$ has order $k>1$ acting canonically on $V=\mathbb{R}^2$. Then the maximally aligned sets with respect to $G = \langle g \rangle$ are the convex cones of angular width $\frac{\pi}{k}$.
\end{example}

\begin{example}
\label{ex.two cones}
    Consider $G=\pm$ acting on $V=\mathbb{R}^3$. Both the nonnegative octant and the ``ice cream cone'' 
    \[
    \Big\{ (x_1,x_2,x_3) \in \mathbb{R}^3 : \sqrt{x_1^2+x_2^2} \leq x_3\Big\}
    \] 
    form maximally aligned sets.
    Despite the fact that both aligned sets are maximal, they are distinct in structure; see  Figure~\ref{fig:aligned_sets} for an illustration.
\end{example}

We suspect it would be nontrivial to find a nice description of all maximally aligned sets, but we \textit{can} say they are all closed convex cones.

\begin{proposition}
For $G \leq O(V)$, every set in $V$ that is maximally aligned with respect to $G$ is a closed convex cone.
\end{proposition}

\begin{proof}
    Suppose $S$ is maximally aligned, and let $C \supseteq S$ denote the topological closure of its conic hull. 
    Since $G$ is orthogonal, the alignment of $S$ is equivalent to
    \[
    \langle x,(I-g)y\rangle \geq 0 \qquad \text{for all $x,y \in S$, $g \in G$.}
    \]
    This inequality is preserved under nonnegative linear combinations and limits in $x$ and $y$, so it holds for any $x,y \in C$ as well.
    Thus $C$ is aligned, and so the maximality of~$S$ gives
    $S=C$.
\end{proof}

\begin{figure}
\centering
\begin{tikzpicture}[
  x=1.50cm,y=1.50cm,
  line cap=round,line join=round,
  every node/.style={text=ink},
  ray/.style={draw=ink,line width=.75pt},
  coordinateaxis/.style={ray,-{Stealth[length=4pt,width=3.5pt]}},
  continuation/.style={coordinateaxis},
  hiddenaxis/.style={ray,dash pattern=on 2.5pt off 2.5pt},
  axislabel/.style={inner sep=4pt}
]
\pgfmathsetmacro{\elevation}{23}
\pgfmathsetmacro{\s}{sin(\elevation)}
\pgfmathsetmacro{\c}{cos(\elevation)}

\begin{scope}[shift={(0,0)}]
  \pgfmathsetmacro{\leftscale}{1.65}
  \pgfmathsetmacro{\azimuth}{-42}
  \pgfmathsetmacro{\ax}{-\leftscale*sin(\azimuth)}
  \pgfmathsetmacro{\ay}{-\leftscale*\s*cos(\azimuth)}
  \pgfmathsetmacro{\bx}{\leftscale*cos(\azimuth)}
  \pgfmathsetmacro{\by}{-\leftscale*\s*sin(\azimuth)}
  \pgfmathsetmacro{\cz}{\leftscale*\c}
  \begin{scope}[shift={(-.76,0)}]
  \coordinate (O) at (0,0);
  \coordinate (A) at (\ax,\ay);
  \coordinate (B) at (\bx,\by);
  \coordinate (C) at (0,\cz);

  \foreach \P/\Q/\strength in {A/C/.17,B/C/.10,A/B/.12}{
    \fill[cone,opacity=\strength] (O)--(\P)--(\Q)--cycle;
    \foreach \j in {0,...,19}{
      \pgfmathsetmacro{\lo}{1+.016*\j}
      \pgfmathsetmacro{\hi}{1+.016*(\j+1)}
      \pgfmathsetmacro{\fade}{\strength*(20-\j)/20}
      \fill[cone,opacity=\fade]
        ($(O)!\lo!(\P)$)--($(O)!\hi!(\P)$)--
        ($(O)!\hi!(\Q)$)--($(O)!\lo!(\Q)$)--cycle;
    }
  }

  \filldraw[fill=section,fill opacity=.43,draw=sectionedge,line width=.9pt]
    (A)--(B)--(C)--cycle;
  \draw[ray] (A)--(O);
  \draw[hiddenaxis] (O)--(B);
  \draw[ray] (O)--(C);

  \draw[continuation] (A)--($(O)!1.43!(A)$)
    node[right,axislabel] {$x_1$};
  \draw[continuation] (B)--($(O)!1.43!(B)$)
    node[right,axislabel] {$x_2$};
  \draw[continuation] (C)--($(O)!1.39!(C)$)
    node[right,axislabel] {$x_3$};
  \fill[ink] (O) circle[radius=.022];
  \end{scope}

  \node at (0,-1.19) {$x_1,x_2,x_3\geq0$};
\end{scope}

\begin{scope}[shift={(4.90,0)}]
  \coordinate (Q) at (0,0);
  \pgfmathsetmacro{\tangentangle}{asin(tan(\elevation))}
  \pgfmathsetmacro{\tangentx}{sqrt(1-tan(\elevation)^2)}
  \pgfmathsetmacro{\tangenty}{\c-\s*tan(\elevation)}
  \coordinate (L) at (-\tangentx,\tangenty);
  \coordinate (R) at (\tangentx,\tangenty);

  \draw[coordinateaxis] (Q)--({.94/sqrt(2)},{-.94*\s/sqrt(2)})
    node[below right,axislabel] {$x_2$};
  \draw[coordinateaxis] (Q)--({-.94/sqrt(2)},{-.94*\s/sqrt(2)})
    node[below left,axislabel] {$x_1$};

  \path[fill=cone,fill opacity=.20]
    (Q)--(R)
    arc[start angle=-\tangentangle,end angle={180+\tangentangle},
      x radius=1,y radius=\s]--cycle;

  \foreach \j in {0,...,19}{
    \pgfmathsetmacro{\lo}{1+.016*\j}
    \pgfmathsetmacro{\hi}{1+.016*(\j+1)}
    \pgfmathsetmacro{\fade}{.20*(20-\j)/20}
    \fill[cone,opacity=\fade]
      ({\lo*\tangentx},{\lo*\tangenty})--
      ({\hi*\tangentx},{\hi*\tangenty})
      arc[start angle=-\tangentangle,end angle={180+\tangentangle},
        x radius=\hi,y radius={\hi*\s}]--
      ({-\lo*\tangentx},{\lo*\tangenty})
      arc[start angle={180+\tangentangle},end angle=-\tangentangle,
        x radius=\lo,y radius={\lo*\s}]--cycle;
  }

  \draw[draw=ink,line width=.4pt]
  ($(Q)!1.4!(L)$)--(Q)--($(Q)!1.4!(R)$);

  \filldraw[fill=section,fill opacity=.53,draw=sectionedge,line width=.9pt]
    (0,\c) ellipse[x radius=1,y radius=\s];

  \draw[hiddenaxis] (Q)--(0,\c);
  \draw[coordinateaxis] (0,\c)--(0,{2.11})
    node[right,axislabel] {$x_3$};

  \fill[ink] (Q) circle[radius=.022];

  \node at (0,-1.19) {$\sqrt{x_1^2+x_2^2}\leq x_3$};
\end{scope}
\end{tikzpicture}
\caption{The nonnegative orthant (left) and the ice cream cone (right) form maximally aligned sets in $V=\mathbb{R}^3$ under the action of $G = \pm$; see Example~\ref{ex.two cones}.
For the sake of illustration, we display cross sections with the planes $x_1+x_2+x_3=1$ and $x_3=1$, respectively.
}
    \label{fig:aligned_sets}
\end{figure}
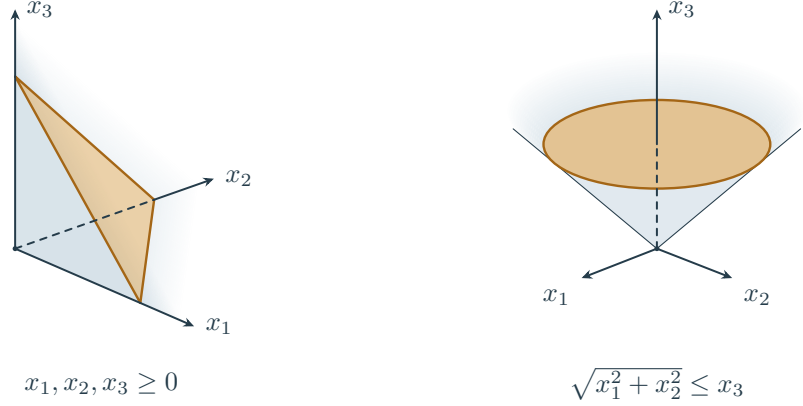

Next, we formalize the idea that changing orbit representatives can transform a non-aligned set into an aligned set.

\begin{definition}
A set $S \subseteq V$ is \textbf{alignable} with respect to $G$ if for every $x \in S$ there exists $g_x \in G$ such that the set $\{ g_x x:x \in S\}$ is aligned.
\end{definition}

Importantly, not every set is alignable.
For example, the points
\[
(2,0), 
\qquad
(-1,\sqrt{3}), 
\qquad
(-1,-\sqrt{3})
\]
are not alignable with respect to $\pm$.
(This is straightforward to verify.)
Meanwhile, when a generic set $S$ \textit{is} alignable, it is easy to find aligned orbit representatives: simply align all of the points to any chosen member of $S$.
For example, the points
\[
(1,0),
\qquad
(-1,-1),
\qquad
(-1,-2)
\]
are not aligned with respect to $\pm$, but they become aligned once we negate the second and third vectors so that they both have a positive inner product with the first vector.
This reduces general alignment to a pairwise alignment problem.
The following examples illustrate how the computational complexity of pairwise alignment is sensitive to the group.

\begin{example}
   \label{ex:cyclic_perm_alignment}
    Let $G=C_d$ act on $V=\mathbb{R}^{d}$ as cyclic permutations of coordinates. 
    Two vectors $x$ and $y$ are aligned precisely when their circular cross correlation peaks at $k=0$. 
    In terms of the discrete Fourier transform $\mathcal{F}$ and entrywise product $\circ$, this is equivalent to $\mathcal{F}^{-1} ( \mathcal{F}(x) \circ \overline{\mathcal{F}(y)} )$ being maximized at the zeroth entry. 
    Thus, the fast Fourier transform offers an efficient method for pairwise alignment in this setting.
\end{example}

\begin{example}
    \label{ex:perm_alignment_2}
    Let $G=S_n$ act on $V=\mathbb{R}^{n \times n}$ by conjugation.
    Any method of pairwise alignment in this setting could be used to decide whether a graph contains a Hamiltonian circuit (by aligning an $n$-cycle's adjacency matrix to that of the graph).
    As such, the alignment problem in this case is $\mathsf{NP}$-hard.
\end{example}

\subsection{Aligned sets have easy-to-compute Fr\'echet means}

Aligned sets correspond to easy instances of geometric problems in the quotient space, since the (typically hard) quotient geometry is identical to the original (easy) Euclidean geometry.
We are now ready to make this explicit:
It is easy to compute the Fr\'echet mean of aligned data.

\begin{theorem}
\label{thm:frechet_aligned}
Suppose $\{x_i\}_{i=1}^n\subset V$ is aligned with respect to $G\leq O(V)$, and let $\mu\in V$ denote its Euclidean mean.
Then $[\mu]\in V/G$ is a Fr\'echet mean of $\{[x_i]\}_{i=1}^n$. 
\end{theorem}

\begin{proof}
First, we recall a standard result about mean squared displacement:
For every $z_1,\ldots,z_n\in V$ with centroid $\mu\in V$, it holds that
\[
\|z_i-z_j\|^2
=\|(z_i-\mu)-(z_j-\mu)\|^2
=\|z_i-\mu\|^2-2\langle z_i-\mu,z_j-\mu\rangle+\|z_j-\mu\|^2,
\]
and so the mean squared displacement has the symmetrized form:
\begin{align*}
&\frac{1}{2n^2}\sum_{i,j=1}^n\|z_i-z_j\|^2\\
&=\frac{1}{2n}\sum_{i=1}^n\|z_i-\mu\|^2-\bigg\langle \frac{1}{n}\sum_{i=1}^n(z_i-\mu),\frac{1}{n}\sum_{j=1}^n(z_j-\mu)\bigg\rangle+\frac{1}{2n}\sum_{j=1}^n\|z_j-\mu\|^2\\
&=\frac{1}{n}\sum_{i=1}^n\|z_i-\mu\|^2.
\end{align*}
We now use this to prove our result.
Since every Karcher mean is a Euclidean mean by Theorem~\ref{thm:euc_mean}, it follows that $F$ is minimized by a point of the form $m=\frac{1}{n}\sum_{j=1}^ng_jx_j$, where $g_1,\ldots,g_n\in G$ together minimize the quantity
\[
F(m)=\frac{1}{n}\sum_{i=1}^n\bigg\|g_ix_i-\frac{1}{n}\sum_{j=1}^ng_jx_j\bigg\|^2
=\frac{1}{2n^2}\sum_{i,j=1}^n\|g_ix_i-g_jx_j\|^2
\geq\frac{1}{2n^2}\sum_{i,j=1}^nd([x_i],[x_j])^2.
\]
Since $\{x_i\}_{i=1}^n$ is aligned by assumption, equality in the lower bound is achieved by $g_1=\cdots=g_n=I$.
Thus, $F$ is minimized by $\mu$, as claimed.
\end{proof}

We note that the converse of Theorem~\ref{thm:frechet_aligned} is far from true.
In particular, given \textit{any} non-alignable $\{y_i\}_{i=1}^n$, select $\mu$ so that $[\mu]$ is the Fr\'echet mean of $\{[y_i]\}_{i=1}^n$.
Since every Fr\'echet mean is a Karcher mean, and furthermore by Theorem~\ref{thm:euc_mean}, every Karcher mean is a Euclidean mean, it follows that there exist $g_1,\ldots,g_n\in G$ for which $\mu=\frac{1}{n}\sum_{i=1}^ng_iy_i$.
Now, define $x_i{\coloneqq}g_iy_i$ for each $i$.
Since $\{y_i\}_{i=1}^n$ was not alignable by assumption, it follows that $\{x_i\}_{i=1}^n$ is not aligned, and yet it satisfies the conclusion of Theorem~\ref{thm:frechet_aligned}.

A generalization of these ideas was studied in~\cite{devilliers2017template}, where the authors consider noisy data generated around a template in $V$, which they project into the quotient space $V/G$.
Under this setup, they investigate the consistency of the Fr\'echet mean, specifically, when does the mean of data in $V$ correspond to the Fr\'echet mean in $V/G$?
By Theorem~\ref{thm:frechet_aligned}, the Fr\'echet mean is consistent if the support of the distribution is aligned, but what if it isn't?
Is it necessary or sufficient for every point in the support of the distribution to be aligned to the mean?
The authors proved necessity under mild conditions, but left sufficiency open.
Below, we present an example that disproves sufficiency.

\begin{example}
We consider an explicit example with $V=\mathbb{R}^2$ and $G=\pm$.
Put
\[
x_1{\coloneqq}(\epsilon,1),
\qquad
x_2{\coloneqq}(\epsilon,-1),
\qquad
x_3{\coloneqq}(3,0).
\]
Each of these points is aligned to the Euclidean centroid $(1+\frac{2}{3}\epsilon,0)$, and yet the Fr\'echet functional is smaller at $(1,\frac{2}{3})$ when $\epsilon>0$ is sufficiently small.
Thus, the Fr\'echet mean is inconsistent despite each $x_i$ being aligned to the Euclidean centroid.
\end{example}

\section{{Synchronization relaxations of the Fr\'echet objective}}
\label{sec:relaxations}

While alignment allows us to compute a Fr\'echet mean,
in this section, we explore how to compute a Fr\'echet mean when data is \textit{not} alignable.
By Theorem~\ref{thm:euc_mean}, computing a Fr\'echet mean is equivalent to finding orbit representatives that are optimal in some sense, so we can express our problem as an optimization over the group.

We state our results for $x_i \in \mathbb{R}^d$ for the sake of readability, though they extend directly to $x_i \in \mathbb{R}^{d\times p}$ with $p>1$ and $G \leq O(d)$ acting on the left.

\begin{theorem}
\label{theorem:rewriting}
Fix any compact $G \leq O(d)$ and $x_1, \dots, x_n \in \mathbb{R}^{d}$.
Every Fr\'echet mean of $\{[x_i]\}_{i=1}^n$ is given by the orbit of $\frac{1}{n} \sum_{i=1}^n g_i x_i$, where $g_1,\ldots,g_n\in G$ together form any solution to the program
\begin{equation}
\label{eqn:block}
\text{maximize}
\qquad
\sum_{i,j=1}^n \langle x_i x_j^\T, g_i^\T g_j \rangle_F
\qquad
\text{subject to}
\qquad
g_1,\ldots,g_n\in G.
    \end{equation}  
\end{theorem}
\begin{proof}
    As in the proof of Theorem~\ref{thm:frechet_aligned}, we seek $g_1, \dots, g_n \in G$ that minimize 
    \[
        \sum_{i,j=1}^n \| g_i x_i - g_j x_j\|^2.
    \]
    Expanding the square, this is equivalent to maximizing the sum of cross terms:
    \[
          \sum_{i,j=1}^n \langle g_ix_i, g_jx_j\rangle 
          = \sum_{i,j=1}^n \langle x_i x_j^\T, g_i^\T g_j \rangle_F.
          \]
\end{proof}

We observe that~\eqref{eqn:block} is similar to the objective found in the \textit{group synchronization problem}~\cite{singerAngularSynchronizationEigenvectors2009, bandeiraTightnessMaximumLikelihood2017}, where the goal is to recover group elements $g_1,\ldots,g_n\in G$ that best fit noisy measurements of pairwise relations $g_{ij} \approx g_i g_j^{-1}$. 
Assuming an orthogonal group action, the synchronization problem can be formalized as
\[
\text{minimize}
\qquad
\sum_{i,j=1}^n \|g_{ij} - g_i g_j^\T\|_F^2
\qquad
\text{subject to}
\qquad
g_1,\ldots,g_n\in G.
\]
Expanding the Frobenius norm and applying the orthogonality of $G$ gives
\begin{equation}
\label{eqn:synch}
\text{maximize}
\qquad
\sum_{i,j=1}^n \langle g_{ij} , g_i g_j^\T\rangle_F
\qquad
\text{subject to}
\qquad
g_1,\ldots,g_n\in G.
\end{equation}
The following subsections exploit the similarity between~\eqref{eqn:block} and~\eqref{eqn:synch} to transfer ideas from group synchronization to Fr\'echet means.
We present spectral and semidefinite methods for the general setting in Sections~\ref{sec:general_spectral} and~\ref{sec:general_SDP}, and then we give specialized methods for $G=\pm$ in Section~\ref{sec:pm} and for general cyclic actions in Section~\ref{sec:cyclic}.

\subsection{Spectral Relaxation}
\label{sec:general_spectral}

Given $M\in\mathbb{R}^{nd\times nd}$, let $M_{[ij]}$ denote the $(i,j)$th $d\times d$ block of $M$.
For $G\leq O(d)$, consider the program
\[
\text{maximize}
\qquad
\sum_{i,j=1}^n \langle M_{[ij]}, g_i^\T g_j \rangle_F
\qquad
\text{subject to}
\qquad
g_1,\ldots,g_n\in G.
\]
For example, the program~\eqref{eqn:block} takes this form once we put $x\coloneqq{}[x_1;\cdots;x_n]\in\mathbb{R}^{nd}$ and $M\coloneqq{}xx^\T$.
We may reformulate this program as
\begin{equation}
\label{eq.reformulated frechet}
\text{maximize}
\quad
\langle M, Y \rangle_F
\quad
\text{subject to}
\quad
Y 
= \begin{bmatrix} g_1^\T \\ \vdots \\ g_n^\T \end{bmatrix} 
\begin{bmatrix} g_1 \dots g_n\end{bmatrix},
\quad
g_1,\ldots,g_n\in G.
\end{equation}
Considering $Y$ is $n$ times a rank-$d$ orthogonal projection matrix, one might consider the spectral relaxation
\begin{equation}
\label{eq.spectral relaxation}
\text{maximize}
\quad
\langle M, nP \rangle_F
\quad
\text{subject to}
\quad
P^2=P^\T=P,
\quad
\operatorname{rank}P=d.
\end{equation}
By the von Neumann trace inequality, the solution to \eqref{eq.spectral relaxation} is given by the orthogonal projection onto the top $d$ eigenvectors of (the symmetric part of) $M$.
Next, we may round a solution $P$ of~\eqref{eq.spectral relaxation} to a feasible point of~\eqref{eq.reformulated frechet} by first writing $Z^\T Z=nP$ for some $Z\in\mathbb{R}^{d\times nd}$, and then for some appropriate choice of $Q\in O(d)$, rounding each $d\times d$ block of $QZ$ to the nearest element of $G$.

Unfortunately, our intended use case takes $M$ to be the rank-$1$ matrix $xx^\T$, thereby making this spectral relaxation degenerate
(though such degeneracy may not be an issue in the setting where $x_i\in\mathbb{R}^{d\times p}$ with $p\geq d$).
To avoid this rank-$1$ degeneracy, we are inclined to pass to a nearby synchronization problem:
Take $M_{[ij]}$ to be the group element closest to $x_ix_j^\T$ in Frobenius norm.
This group element $g_{ij}$ happens to be the $g\in G$ that maximizes $\langle x_i,gx_j\rangle$, i.e., that aligns $x_j$ to $x_i$.
As a more refined approach, one might take $M_{[ij]}$ to be a \textit{nonnegative multiple} of a group element that is closest to $x_ix_j^\T$ in Frobenius norm; modulo a constant factor, this takes $M_{[ij]}$ to be the positive part $\langle x_i,g_{ij}x_j\rangle_+$ times the closest group element $g_{ij}$.
See Algorithm~\ref{alg:spectral_general} for a complete description of our spectral relax-and-round algorithm.

\begin{algorithm}[ht]
\caption{Spectral Method for the Fr\'echet Mean in $V/G$}
\label{alg:spectral_general}
\begin{algorithmic}[1]
\State \textbf{Input:} $x_1,\dots,x_n \in \mathbb{R}^{d \times p}$, finite $G \le O(d)$, $T_{\text{max}} \in \mathbb{N}$
\State \textbf{Output:} Member of $\mathbb{R}^{d \times p}$ whose $G$-orbit is an approximate Fr\'echet mean

\Statex

\If{$p < d$}
    \State \textbf{Synchronization Approximation:}
    \For{$i=1,\dots,n$}
        \State $M_{ii} = \|x_i\|_F^2 I_d$
    \EndFor
    \For{$1 \le i < j \le n$}
        \State $g_{ij} \in \arg\max_{g\in G} \langle x_i, gx_j\rangle$
        \State $M_{ij} = \langle x_i, g_{ij}x_j\rangle_+\, g_{ij}$
        \State $M_{ji} = M_{ij}^\T $
    \EndFor
\Else
    \State $\mathbf{x} = [x_1;\dots;x_n] \in \mathbb{R}^{nd \times p}$
    \State $M = \mathbf{x}\mathbf{x}^\T $
\EndIf

\State Compute rank-$d$ spectral factor $Y \in \mathbb{R}^{nd \times d}$ of $M$

\Statex
\State \textbf{Procrustes-Based Rounding to $G^n$}

\State Let $Y_1$ be the first $d\times d$ block of $Y$
\State Compute SVD $Y_1^\T  = U\Sigma V^\T $, set $W = UV^\T $

\For{$i=1$ to $n$}
    \State $g_i \in \arg\min_{g\in G} \|Y_i W - g^\T \|_F$
\EndFor

\For{$t=1$ to $T_{\text{max}}$}
    \State $C = \sum_{i=1}^n Y_i^\T  g_i^\T$
    \State Compute SVD $C = U\Sigma V^\T $, set $W_{\text{new}} = UV^\T $

    \For{$i=1$ to $n$}
        \State $g_i \in \arg\min_{g\in G} \|Y_i W_{\text{new}} - g^\T \|_F$
    \EndFor

    \State $W \gets W_{\text{new}}$
\EndFor

\State \Return $\frac{1}{n} \sum_{i=1}^n g_i x_i$
\end{algorithmic}
\end{algorithm}

\subsection{Semidefinite Relaxation}
\label{sec:general_SDP}

We return to the program~\eqref{eq.reformulated frechet} with $M=xx^\T$, the feasibility region of which consists of all $Y\in\mathbb{R}^{nd\times nd}$ such that
\[
Y_{[ii]}=I_d \text{ for all }i,
\qquad
Y_{[ij]}\in G \text{ for all }i,j,
\qquad
Y\succeq0,
\qquad
\operatorname{rank}(Y)=d.
\]
To obtain a convex relaxation, we discard the rank constraint and replace $Y_{[ij]}\in G$ with $Y_{[ij]}\in\conv{G}$, the convex hull of $G\subseteq\mathbb{R}^{d\times d}$:

\medskip
\noindent\textbf{Fréchet Mean SDP:}
\begin{equation} 
\label{eqn:fullSDP}
\begin{split}
    \text{maximize} \qquad & \langle xx^\T, Y \rangle_F \\
    \text{subject to} \qquad & 
   Y_{[ii]}=I_d \text{ for all }i,
\quad
Y_{[ij]}\in \conv{G} \text{ for all }i,j,
\quad
Y\succeq0.   
\end{split}
\end{equation}
Unlike the spectral relaxation, the semidefinite relaxation is frequently tight in practice, meaning the solution to the SDP is feasible in~\eqref{eq.reformulated frechet} and exactly solves~\eqref{eqn:block}. 
When the relaxation is not tight, we may apply a rounding procedure to extract a feasible solution to the original problem, much like in the spectral method.

Taking a step back, we note that, given $\{x_i\}_{i=1}^n\subset\mathbb{R}^d$ and $\{g_i\}_{i=1}^n\subset G$, it is sometimes computationally feasible to determine whether the $G$-orbit of $\frac{1}{n}\sum_{i=1}^n g_ix_i$ is the Fr\'echet mean of $\{[x_i]\}_{i=1}^n$.
For example, this occurs when $\{g_ix_i\}_{i=1}^n$ is aligned as in Section~\ref{sec:alignment}.
Alternatively, this holds when the matrix $Y$ defined by $Y_{[ij]}=g_i^\T g_j$ solves the Fr\'echet mean SDP~\eqref{eqn:fullSDP}.
In general, given a candidate solution to a semidefinite program, one may certify its optimality using the \textit{dual program}.
In Appendix~\ref{sec:dual}, we prove the following:

\begin{lemma}
\label{lem.dual}
The dual program of~\eqref{eqn:fullSDP} is
\label{lemma:dual}
\begin{equation}
\label{eqn:sdp_dual}
\min_{W \succeq xx^\T }\,
\max_{H \in \mathcal{H}}\,
\langle H,W\rangle_F,
\end{equation}
where $\mathcal{H}$ consists of all symmetric $H\in\mathbb{R}^{nd\times nd}$ such that
\[
H_{[ii]}=I_d
\text{ for all } i,
\qquad
H_{[ij]}\in G
\text{ for all } i,j.
\]
\end{lemma}

Given a candidate solution $(g_1,\ldots,g_n)$ to~\eqref{eqn:block}, then by renaming $x_i\leftarrow g_ix_i$, we may establish the optimality of $(g_1,\ldots,g_n)$ by showing that $\1 \1^\T  \otimes I_d$ is a solution to the Fr\'echet mean SDP~\eqref{eqn:fullSDP}.
In what follows, we obtain a dual certificate to this effect by appealing to Lemma~\ref{lem.dual}.

\begin{theorem} 
\label{thm:sdp_cert}
    Suppose there exists a $W \in \mathbb{R}^{nd \times nd}$ such that
\begin{equation}\label{eqn:sdp_cert}
\begin{gathered}
    W \succeq xx^\T,
    \qquad
    I_d \in \argmax_{g \in G} \langle g, W_{[ij]} \rangle_F
    \quad \text{for all } i<j,
    \\
    \langle \1\1^\T \otimes I_d, W\rangle_F
    = \langle \1\1^\T \otimes I_d, xx^\T\rangle_F.
\end{gathered}
\end{equation}
    Then the $G$-orbit of the mean of $\{x_i\}_{i=1}^n\subset\mathbb{R}^d$ is a Fr\'echet mean of $\{[x_i]\}_{i=1}^n \subset \mathbb{R}^d/G$.
\end{theorem}
\begin{proof}
Suppose there exists such a $W$.
The first condition in~\eqref{eqn:sdp_cert} makes $W$ dual feasible. 
Since every $H\in\mathcal H$ has diagonal blocks $I_d$,
the second condition determines the dual value at $W$:
\[
\max_{H \in \mathcal{H}} \, \langle H, W\rangle_F 
= \langle \1 \1^\T  \otimes I_d, W\rangle_F.
\]
The third condition in~\eqref{eqn:sdp_cert} 
gives that
the primal value at $Y \coloneqq{} \1\1^\T  \otimes I_d$ equals the dual value at $W$. 
By weak duality, 
it follows that $Y$ is primal optimal.
Since~\eqref{eqn:fullSDP} is a relaxation of~\eqref{eqn:block}, we conclude that $g_1=\cdots=g_n=I_d$ solves the latter problem, and the result follows.
\end{proof}

When $\{x_i\}_{i=1}^n$ is aligned, one may take $W=xx^\T$ in Theorem~\ref{thm:sdp_cert}.
Indeed, the first and third conditions of~\eqref{eqn:sdp_cert} are immediate, while the second condition concerns the quantity
\[
\langle g,W_{[ij]}\rangle_F
=\langle g,x_ix_j^\T\rangle_F
=\langle x_i,gx_j\rangle,
\]
which, for aligned data, is maximized by $g=I_d$.
When $\{x_i\}_{i=1}^n$ is not aligned, then $W=xx^\T$ does not satisfy the second condition of~\eqref{eqn:sdp_cert}, but sometimes, an appropriate perturbation does.
We take this approach to prove the following weaker guarantee.

\begin{corollary}
\label{cor:alignment_L}
Fix any finite nontrivial $G\leq O(d)$ and $\{x_i\}_{i=1}^n\subset\mathbb{R}^d$.
Define the signed adjacency matrix $A \in \mathbb{R}^{n \times n}$ by 
\[
A_{ij} 
= \min_{g \in G\setminus\{I_d\}}\frac{\lr{x_i}{x_j} - \lr{x_i}{g x_j}}{d - \Tr{g}},
\]
and let $L$ denote the corresponding graph Laplacian $L = \diag{(A\1)} - A$. 
If $L \succeq 0$, then the $G$-orbit of the mean of $\{x_i\}_{i=1}^n\subset\mathbb{R}^d$ is a Fr\'echet mean of $\{[x_i]\}_{i=1}^n \subset \mathbb{R}^d/G$. 
\end{corollary}
\begin{proof}
By Theorem~\ref{thm:sdp_cert}, it suffices to construct $W$ satisfying~\eqref{eqn:sdp_cert}. Define 
\[
W \coloneqq{} \left(L \otimes I_d\right) + xx^\T .
\]
The first condition in~\eqref{eqn:sdp_cert} follows from our assumption that $L\succeq0$.
For the second condition, fix $i<j$
and $g\in G\setminus\{I_d\}$.
Then
\begin{align*}
\langle I_d,W_{[ij]}\rangle_F
-\langle g, W_{[ij]}\rangle_F
&=\langle I_d-g,x_ix_j^\T-A_{ij}I_d\rangle_F\\
&=\langle x_i,x_j\rangle-\langle x_i,gx_j\rangle-A_{ij}(d-\operatorname{Tr}g)
\geq0,
\end{align*}
where the inequality follows from the definition of $A$.
This establishes the second condition in~\eqref{eqn:sdp_cert}.
Finally, the third condition follows from the fact that $\1\in\ker{L}$:
\[
\langle W-xx^\T , \1\1^\T  \otimes I_d \rangle_F
=\langle L\otimes I_d,\1\1^\T  \otimes I_d \rangle_F
=\langle L,\1\1^\T\rangle_F \langle  I_d, I_d \rangle_F
=0.
\]
Overall, $W$ satisfies~\eqref{eqn:sdp_cert}, and so we are done.
\end{proof}

Corollary~\ref{cor:alignment_L} gives a closed-form spectral test to certify that the orbit of the Euclidean mean is a Fr\'echet mean; this is inspired by the \textit{probably certifiably correct} algorithms introduced in~\cite{bandeira2016note}.
The underlying signed adjacency matrix is also interpretable to some extent.
Indeed, since $g$ is an isometry, we have
\[
\langle x_i,x_j\rangle - \langle x_i,g x_j\rangle
= \frac{1}{2} \big( \|x_i - gx_j\|^2 - \|x_i - x_j\|^2 \big),
\]
and so $A_{ij}$ is nonnegative precisely when $x_i$ and $x_j$ are aligned.
More generally, we can think of $A_{ij}$ as quantifying the alignment between $x_i$ and $x_j$.
If $\{x_i\}_{i=1}^n$ is aligned, then $A$ is entrywise nonnegative, and so $L\succeq0$.
Meanwhile, this positive semidefiniteness will survive a small amount of pairwise misalignment.
See Figure~\ref{fig:pm_certifiable_example} for an illustration of this phenomenon and Section~\ref{sec.case study} for a theoretical treatment.

\begin{figure}
\centering
\includegraphics[width=\linewidth]{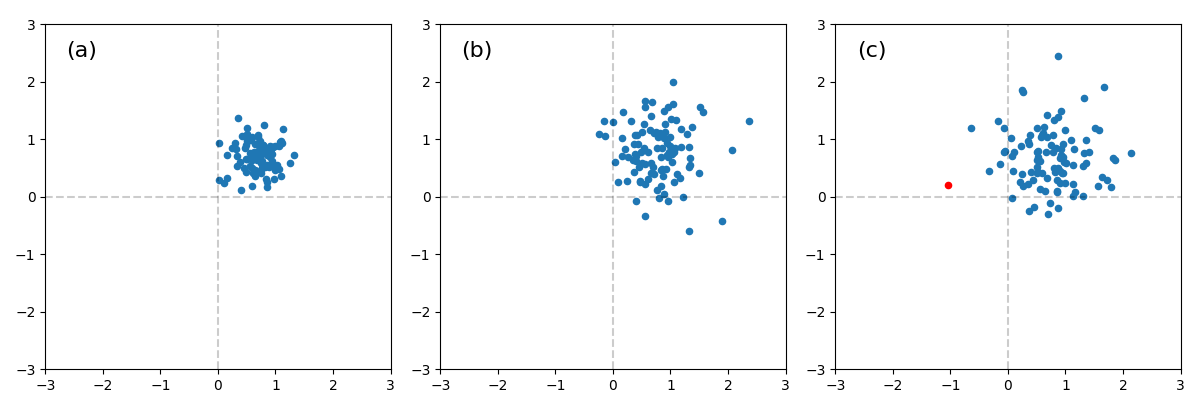}
\caption{Three random draws of 100 Gaussian data points in $\mathbb{R}^2$. 
In (a), the data is aligned with respect to $G=\pm$. 
In (b), the data is not aligned, but the spectral test in Corollary~\ref{cor:alignment_L} certifies that the Fr\'echet mean still coincides with the Euclidean mean. 
In (c), the Fr\'echet mean in $\mathbb{R}^2/\pm$ does not coincide with the Euclidean mean, as the red point must first be negated.
See Section~\ref{sec.case study} for a theoretical treatment of this phenomenon.}
\label{fig:pm_certifiable_example}
\end{figure}

\subsection{Relaxations for Sign Invariance}
\label{sec:pm}

In the special case where $G=\pm$, the Fr\'echet objective~\eqref{eqn:block} can be reformulated as
\begin{equation}
\label{eqn:frechet_pm}
\text{maximize}
\qquad
\langle X^\T X,\epsilon\epsilon^\T\rangle_F
\qquad
\text{subject to}
\qquad
\epsilon_1,\ldots,\epsilon_n\in\{\pm1\},
\end{equation}
where the $i$th column of $X \in \mathbb{R}^{d \times n}$ is $x_i$.
This formulation specializes~\eqref{eq.reformulated frechet} by exploiting the fact that $G$ acts by scalars.\footnote{Indeed, viewing $x_i$ as a member of $\mathbb{R}^{1\times d}$ with $G=O(1)$ acting on the left, then $x=[x_1;\cdots;x_n]=X^\T$ and $xx^T =X^TX$.}

We state our results for the finite group $G=\pm$, although the same scalar formulation has a complex analog with $G=S^1$ acting by scalar multiplication on $V=\mathbb{C}^d$.
For normalized landmark data, this complex Fr\'echet mean problem is a special case of Generalized Procrustes Analysis (GPA)~\cite{gower1975generalized}. 
The connection between Procrustes alignment and Fr\'echet means has been noted previously~\cite{huckemann2012meaning} and has been approached using spectral and semidefinite methods in~\cite{goodall1991procrustes, pizarro2011global}.
These problems also resemble \textit{angular synchronization}~\cite{singerAngularSynchronizationEigenvectors2009} and \textit{$O(1)$-synchronization}~\cite{abbe2014decoding}, which can be approached similarly.
With this in mind, we develop relaxations of~\eqref{eqn:frechet_pm}.

\subsubsection{Spectral Relaxation for Sign Invariance}
\label{sec:pm_spec}

We obtain a spectral relaxation by changing the constraint in~\eqref{eqn:frechet_pm} to $\|\epsilon\|^2=n$.
The solution to this relaxation is then the top eigenvector of the Gram matrix $X^\T X$.
Rounding this top eigenvector~$v$ to a feasible point of~\eqref{eqn:frechet_pm} then gives the approximate solution $\epsilon_i = \sign{v_i}$.

\subsubsection{Semidefinite Relaxation for Sign Invariance}
\label{sec:pm_sdp}

Writing $Z=\epsilon\epsilon^\T$, 
the feasibility region of~\eqref{eqn:frechet_pm} consists of all $Z\in\mathbb{R}^{n\times n}$ such that
\[
\operatorname{diag}Z=\1,
\qquad
Z\succeq0,
\qquad
\operatorname{rank}Z=1.
\]
By dropping the rank-$1$ constraint, we obtain the following:

\medskip
\noindent\textbf{Sign-Invariant Fr\'echet Mean SDP:}
\begin{equation}
\label{eqn:pm_sdp}
\text{maximize}
\qquad
\langle X^\T X,Z\rangle_F
\qquad
\text{subject to}
\qquad
\operatorname{diag}Z=\1,
\qquad
Z\succeq0.    
\end{equation}
When the solution to~\eqref{eqn:pm_sdp} has rank~$1$, it also solves the original problem~\eqref{eqn:frechet_pm}. 
Otherwise, we can round to the feasibility region of~\eqref{eqn:frechet_pm} 
by taking the entrywise sign of the top eigenvector of the SDP solution.

When $Z = \1\1^\T $ is optimal for \eqref{eqn:pm_sdp}, the Euclidean mean determines a Fr\'echet mean.
Similar to Theorem~\ref{thm:sdp_cert}, one may use duality to characterize this situation:

\begin{proposition}[Lemma~4.4 in~\cite{bandeiraTightnessMaximumLikelihood2017}]
\label{prop:sdp_aligned}
The matrix $Z=\1\1^\T$ optimizes~\eqref{eqn:pm_sdp} if and only if $\diag{(X^\T X\1)} - X^\T X \succeq 0$.
\end{proposition}

Observe that the ($\Leftarrow$) direction follows from Corollary~\ref{cor:alignment_L} since $A=\frac{1}{d}X^\T X$, while the ($\Rightarrow$) direction is unique to the sign-invariant setting.

\subsection{Relaxations for Cyclic Actions}
\label{sec:cyclic}

In this section, we consider $G=C_k$ acting orthogonally on $\mathbb{R}^d$ and on the left of $\mathbb{R}^{d\times p}$.
For example, the group might act by cyclic permutations, as in shape analysis (see Section~\ref{sub.mnist}) or time series data (see Section~\ref{sec:applications}).
Sign-invariant data can also be viewed as a special case of this setting, and the following relaxations reduce to those previously discussed in Section~\ref{sec:pm}.

Letting $g$ denote (the representation of) a generator of $C_k$, diagonalizing over~$\mathbb{C}$ gives $g=UDU^*$ for some unitary matrix $U$ and some diagonal matrix $D$ whose diagonal entries $\{\omega_m\}_{m=1}^d$ are $k$th roots of unity.
For each $i\in\{1,\ldots,n\}$, we are given $x_i\in\mathbb{R}^{d\times p}$.
Transform the data as $z_i\coloneqq{}U^*x_i$.\footnote{When $G$ acts by cyclic permutations, we may take $U^*$ to be the discrete Fourier transform.}
Then given $g_1,\ldots,g_n\in C_k$, we write $g_i=UD^{k_i}U^*$ for some $k_i\in\{0,\ldots,k-1\}$, and the objective~\eqref{eqn:block} becomes
\[
\sum_{i,j=1}^n \langle x_i x_j^\T, g_i^\T g_j \rangle_F
=\sum_{i,j=1}^n \langle z_i z_j^*, D^{k_j-k_i} \rangle_F
=\sum_{i,j=1}^n \sum_{m=1}^d \overline{(z_iz_j^*)_{mm}} \, \omega_m^{k_j-k_i}.
\]
Writing $\mathcal{Z}_m
\coloneqq{}[(z_i z_j^*)_{mm}]_{i,j=1}^n$ for each $m\in\{1,\ldots,d\}$, then~\eqref{eqn:block} reduces to
\begin{equation}
\label{eqn:cyclic_frechet}
\text{maximize}
\quad
\sum_{m=1}^d
\big\langle
\mathcal{Z}_m,
[\omega_m^{k_j-k_i}]_{i,j=1}^n
\big\rangle_F
\quad
\text{subject to}
\quad
k_1,\ldots,k_n\in \mathbb{Z}/k\mathbb{Z}.
\end{equation}
Overall, the cyclic structure reduces our $nd \times nd$ matrix problem to an equivalent $n \times n \times d$ tensor problem.
The corresponding spectral and semidefinite relaxations also take special forms.

\subsubsection{Spectral Relaxation for Cyclic Actions}
\label{sec:cyclic_spectral_algo}

The form of~\eqref{eqn:cyclic_frechet} allows us to use ideas from \textit{multi-frequency phase synchronization}~\cite{gaoMultiFrequencyPhaseSynchronization}, where global alignment is obtained by first synchronizing slices $\mathcal{Z}_m$ of a common mode $\omega_m$, then combining the modes to estimate pairwise relative shifts, then finally synchronizing them globally:
\begin{enumerate}
\medskip
\item
For each mode $\omega\in\{ \omega_1, \dots, \omega_d\}$, let $\Omega_\omega$ denote the best rank-$1$ approximation to the positive semidefinite matrix
\[
\sum_{\substack{m\in\{1,\ldots,d\}\\\omega_m = \omega}} \mathcal{Z}_m,
\]
i.e., $\Omega_\omega=\lambda_\omega u_\omega u_\omega^*$, where $\lambda_\omega$ and $u_\omega$ are the top eigenvalue and eigenvector.
\medskip
\item 
For each $i,j\in\{1,\ldots,n\}$ with $i<j$, find $k_{ij}\in\{0,\ldots,k-1\}$ that maximizes
\[
\operatorname{Re}\bigg(
\sum_{\omega \in \{ \omega_1, \dots, \omega_d\}}  (\Omega_\omega)_{ij} \,\omega^{k_{ij} } \bigg).
\]
\medskip
\item 
Estimate $k_1, \dots, k_n$ by angular synchronization:
Define $H \in \mathbb{C}^{n \times n}$ by $H_{ii} =1$ and by $H_{ij} \coloneqq{} e^{2\pi\mathrm{i}k_{ij}/k}$ and $H_{ji}=\overline{H_{ij}}$ for $i < j$, compute the top eigenvector~$v$ of $H$, and set
\[
k_i 
\coloneqq{} \operatorname{round}\big( \tfrac{k}{2\pi} \arg(v_i) \big), 
\qquad 
i\in\{1,\ldots,n\}.
\]
\end{enumerate}
While the first step approximately solves 
\[
\Omega_\omega \approx \begin{bmatrix}
1 & \omega^{k_2 - k_1} & \cdots & \omega^{k_n - k_1} \\
\omega^{k_1 - k_2} & 1 & \cdots & \omega^{k_n - k_2} \\
\vdots & \vdots & \ddots & \vdots \\
\omega^{k_1 - k_n} & \omega^{k_2 - k_n} & \cdots & 1
\end{bmatrix},
\]
we do not round the entries of $\Omega_\omega$ to phases since this would discard magnitude information that benefits step~2.
In the special case where $p=1$ and no mode is repeated (e.g., when $k=d$ and $G$ acts by cyclic permutations), we note that each $\mathcal{Z}_m$ already has rank at most $1$, and so the first step is unnecessary.

\subsubsection{Semidefinite Relaxation for Cyclic Actions}
\label{subsub.sdp cyclic}

Our slices of data form a tensor $\mathcal{Z}\in\mathbb{C}^{n\times n\times d}$ defined by $\mathcal{Z}_{i,j,m}=(\mathcal{Z}_m)_{i,j}$, and the relevant modes are given by the vectors
\[
\chi_t
\coloneqq{}\operatorname{diag}(D^t )
=(\omega_1^t ,\ldots,\omega_d^t ),
\qquad
t\in\{0,\ldots,k-1\}.
\]
This notation yields an equivalent form of the SDP~\eqref{eqn:fullSDP} for cyclic actions:

\medskip
\noindent\textbf{Cyclic Action Fr\'echet Mean SDP:}
\begin{equation}
\label{eqn:cyclic_SDP}
\begin{split}
\text{maximize}
\qquad
&\operatorname{Re}\langle \mathcal{Z},\Omega\rangle\\
\text{subject to} \qquad 
    & \Omega_{i,i,\ell} = 1 \text{ for all } i,\ell,\\
    & \Omega_{i,j,:} \in \conv\{\chi_t\}_{t=0}^{k-1} \text{ for all }i,j, \\
    & {\Omega}_{:,:,\ell} \succeq 0  \text{ for all } \ell.
\end{split}
\end{equation}
This form replaces an $nd \times nd$ semidefinite constraint with $d$ constraints of size $n \times n$, potentially reducing some computational burden.
\section{Tightness of the Sign-Invariant Fr\'echet Mean SDP}
\label{sec.case study}

In this section, we compare alignment with Proposition~\ref{prop:sdp_aligned} for datasets in $\mathbb{R}^d/\pm$.
We call a dataset \emph{SDP aligned} if it satisfies the condition in Proposition~\ref{prop:sdp_aligned}.
Alignment implies SDP alignment, which in turn guarantees that the orbit of the Euclidean mean is a Fr\'echet mean.
The following proposition compares these conditions under a random data model: the noise threshold for alignment decreases as the ambient dimension $d$ grows, while the threshold for SDP alignment remains constant (up to a log factor).
Thus, alignment exhibits a curse of dimensionality that the SDP avoids.

\begin{figure}[t]
\centering
\begin{tikzpicture}[
  tick label/.style={anchor=base, inner sep=3pt},
  pointer/.style={-{Stealth[length=4pt,width=3pt]}, thin}
]
  \def\a{-0.4}
  \def\b{0.4}
  \def\c{4.1}
  \def\h{0.65}

  \fill[green!75!black, fill opacity=0.25]
    (-5,0) rectangle (\a,\h);
  \fill[red, fill opacity=0.25]
    (\b,0) rectangle (5,\h);
  \fill[blue, fill opacity=0.25]
    (-5,-\h) rectangle (\c,0);

  \node[text=green!40!black] at (-3.5,{\h/2}) {\phantom{SDP }aligned};
  \node[text=red!65!black] at ({(\b+5)/2},{\h/2}) {not aligned};
  \node[text=blue!55!black] at (-3.5,{-\h/2}) {SDP aligned};

  \draw[{Stealth}-{Stealth}, thick] (-5.2,0) -- (5.2,0)
    node[right=3pt] {$\sigma^2$};
  \foreach \x in {\a,\b,\c} {
    \draw[thick] (\x,-0.11) -- (\x,0.11);
  }

  \coordinate (figureCenter) at (current bounding box.center);
  \matrix[anchor=base, inner sep=0pt, outer sep=0pt,
          column sep=1.65cm, xshift=-1.9pt]
    at (figureCenter |- 0,-1.55) {
    \node[tick label] (labelA)
      {$\displaystyle\frac{c}{\sqrt{d\log(n^2/\epsilon)}}$}; &
    \node[tick label] (labelB)
      {$\displaystyle\frac{C}{\sqrt{d}}$}; &
    \node[tick label] (labelC)
      {$\displaystyle\frac{c}{\log(n^2/\epsilon)}$}; \\
  };

  \draw[pointer] (labelA.north) to[out=90,in=-130,looseness=0.7] (\a,-0.14);
  \draw[pointer] (labelB.north) to[out=90,in=-50] (\b,-0.14);
  \draw[pointer] (labelC.north) to[out=90,in=-130] (\c,-0.14);
\end{tikzpicture}
\caption{Illustration of Proposition~\ref{prop.bounds in figure of merit} in the usual setting where $d\gg \log^2(n^2/\epsilon)$. This establishes the strength of the sign-invariant Fr\'echet mean SDP~\eqref{eqn:pm_sdp} over the notion of alignment introduced in Section~\ref{sec.alignment}.}
\label{fig:align_comp}
\end{figure}
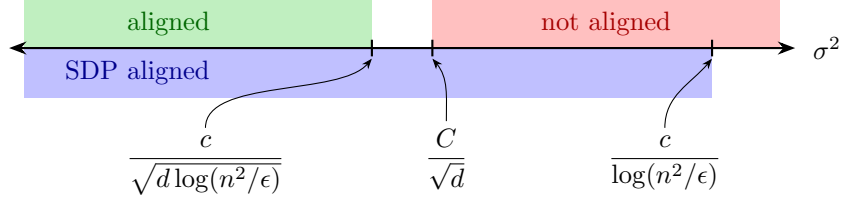

\begin{proposition}
\label{prop.bounds in figure of merit}
There exist universal constants $C,c>0$ for which the following holds:
For every $\epsilon\in(0,1/2)$ and $d,n\in\mathbb{N}$ satisfying
\[
C\log(n^2/\epsilon)
\leq d
\leq n,
\]
if we fix $x\in S^{d-1}$ and draw iid random vectors $\{x_i\}_{i=1}^n$ from $N(x,\sigma^2I_d)$,
then
\[
\begin{aligned}
\sigma^2&\leq\frac{c}{\sqrt{d\log(n^2/\epsilon)}}
&\qquad\Longrightarrow\qquad
&&\mathbb{P}(\{x_i\}_{i=1}^n\text{ aligned})
&>1-\epsilon,\\
\sigma^2
&\geq\frac{C}{\sqrt{d}}
&\qquad\Longrightarrow\qquad
&&\mathbb{P}(\{x_i\}_{i=1}^n\text{ not aligned})
&>1-\epsilon,\\[0.5em]
\sigma^2
&\leq\frac{c}{\log(n^2/\epsilon)}
&\qquad\Longrightarrow\qquad
&&\mathbb{P}(\{x_i\}_{i=1}^n\text{ SDP aligned})
&>1-\epsilon.
\end{aligned}
\]
\end{proposition}

This reveals a sizable gap between alignment and SDP alignment.
When
\[
    \frac{1}{\sqrt{d}}
    \lesssim \sigma^2
    \lesssim \frac{1}{\log(n^2/\epsilon)},
\]
it holds with high probability that
Gaussian points are SDP aligned but not aligned; see Figure~\ref{fig:align_comp} for an illustration.
Thus, for high-dimensional sign-invariant data, the SDP tolerates substantially more noise than alignment does.
Experiments in Section~\ref{sec:numerics} show a similar gap for other groups.

\begin{proof}[Proof of Proposition~\ref{prop.bounds in figure of merit}]
We prove each claimed implication separately.

\medskip
\noindent
\textbf{Claim I:}
Fix $i,j\in\{1,\ldots,n\}$ with $i\neq j$.
We start by upper bounding the probability that $x_i$ and $x_j$ are not aligned, i.e., $\langle x_i,x_j\rangle<0$.
If we denote $z_i{\coloneqq}x_i-x$, $z_j{\coloneqq}x_j-x$, $W{\coloneqq}\langle z_i,x\rangle+\langle x,z_j\rangle$ and $Z{\coloneqq}\langle z_i,z_j\rangle$,
then
\[
\p{\langle x_i,x_j\rangle<0}
=\p{1+W+Z<0}
\leq\p{W<-\tfrac{1}{2}}+\p{Z<-\tfrac{1}{2}}.
\]
Since $W\sim N(0,2\sigma^2)$, the Chernoff bound gives $\p{W<-\tfrac{1}{2}}\leq e^{-1/(16\sigma^2)}$.
Meanwhile, an application of the Hanson--Wright inequality gives
\[
\p{Z<-\tfrac{1}{2}}
\leq 2\operatorname{exp}\big(-c_1\min\{\tfrac{1}{d\sigma^4},\tfrac{1}{\sigma^2}\}\big),
\]
where $c_1>0$ is a universal constant.
We combine these estimates to obtain
\[
\p{\langle x_i,x_j\rangle<0}
\leq 3\operatorname{exp}\big(-c_2\min\{\tfrac{1}{d\sigma^4},\tfrac{1}{\sigma^2}\}\big),
\]
where $c_2>0$ is a universal constant.
An application of the union bound then gives
\[
\mathbb{P}(\{x_i\}_{i=1}^n\text{ aligned})
\geq 1-\tbinom{n}{2}\cdot 3\operatorname{exp}\big(-c_2\min\{\tfrac{1}{d\sigma^4},\tfrac{1}{\sigma^2}\}\big).
\]
Squaring and rearranging our hypothesis on $\sigma^2$ gives $\tfrac{1}{d\sigma^4}
\geq\tfrac{1}{c^2}\log(n^2/\epsilon)$.
Meanwhile, combining our hypothesis on $\sigma^2$ with the assumption $d\geq C\log(n^2/\epsilon)$ gives
\[
\frac{1}{\sigma^2}
\geq\frac{\sqrt{d\log(n^2/\epsilon)}}{c}
\geq\frac{\sqrt{C}}{c}\cdot\log(n^2/\epsilon).
\]
Taking $C$ sufficiently large and $c$ sufficiently small, it follows that our lower bound on the aligned probability is strictly greater than $1-\epsilon$, as desired.

\medskip
\noindent
\textbf{Claim II:}
Fix $i,j\in\{1,\ldots,n\}$ with $i\neq j$.
We start by lower bounding the probability that $x_i$ and $x_j$ are not aligned.
Recalling $z_i$, $z_j$, $W$, and $Z$ from the previous argument, note that simultaneously negating $z_i$ and $z_j$ negates $W$ and fixes $Z$.
It follows that
\[
\p{\langle x_i,x_j\rangle<0}
\geq \p{W<0,Z<-1}
= \tfrac{1}{2}\p{Z<-1}
= \tfrac{1}{4}\p{Z^2>1},
\]
where the last step uses the fact that the distribution of $Z$ is symmetric about zero.
We obtain a further lower bound using the Paley--Zygmund inequality.
First, we compute the relevant moments:
\begin{align*}
\mathbb{E}Z^2
&%
=\mathbb{E}\bigg[\sum_{k,l=1}^d(z_i)_k(z_j)_k(z_i)_l(z_j)_l\bigg]
=\sum_{k=1}^d\mathbb{E}[(z_i)_k^2]\,\mathbb{E}[(z_j)_k^2]
=d\sigma^4,\\
\mathbb{E}Z^4
&=\mathbb{E}\big[\mathbb{E}[Z^4|z_i]\big]
=\mathbb{E}\big[3\sigma^4\|z_i\|^4\big]
=3\sigma^4\cdot\mathbb{E}\bigg[\sum_{k,l=1}^d(z_i)_k^2(z_i)_l^2\bigg]
=3d(d+2)\sigma^8.
\end{align*}
Then we combine the Paley--Zygmund inequality with the fact that $d\geq 1$ to obtain
\[
\p{Z^2>\tfrac{1}{2}\mathbb{E}Z^2}
\geq\frac{(\mathbb{E}Z^2)^2}{4\mathbb{E}Z^4}
=\frac{d}{12(d+2)}
\geq\frac{1}{36}.
\]
Provided $C\geq\sqrt{2}$, our hypothesis on $\sigma^2$ gives $\frac{1}{2}\mathbb{E}Z^2 
= \frac{1}{2}d\sigma^4
\geq 1$, and so
\[
\p{\langle x_i,x_j\rangle<0}
\geq \tfrac{1}{4}\p{Z^2>1}
\geq \tfrac{1}{4}\p{Z^2>\tfrac{1}{2}\mathbb{E}Z^2}
\geq\tfrac{1}{144}.
\]
Now note that $\{x_i\}_{i=1}^n$ is  aligned only if $\langle x_i,x_{i+1}\rangle\geq0$ for each odd $i<n$.
These events are independent and have the same probability, and so
\[
\mathbb{P}(\{x_i\}_{i=1}^n\text{ aligned})
\leq \Big(1-\mathbb{P}(\langle x_1,x_2\rangle<0)\Big)^{\lfloor n/2\rfloor}
\leq \big(\tfrac{143}{144}\big)^{\lfloor n/2\rfloor}.
\]
Since $n\geq C\log(n^2/\epsilon)\geq C\log(1/\epsilon)$ by assumption, taking $C$ sufficiently large makes this upper bound strictly less than $\epsilon$, as desired.

\medskip
\noindent
\textbf{Claim III:}
Following Proposition~\ref{prop:sdp_aligned}, we put
$L
{\coloneqq}\operatorname{diag}(X^\T X\1)-X^\T X$,
where $X$ has columns $x_1,\ldots,x_n$. 
It suffices to show $L\succeq0$.
Define
\[
\mu{\coloneqq}\frac1n\sum_{i=1}^n x_i,
\qquad
r_i{\coloneqq}x_i-\mu,
\qquad
R{\coloneqq}\left[\begin{array}{ccc}r_1&\cdots&r_n\end{array}\right].
\]
Since $L\1=0$, we may focus on the unit vectors
$v$ orthogonal to $\1$. 
Since $X\1=n\mu$ and $Xv=Rv$ for every such vector, we have
\begin{align*}
v^\T Lv
=v^\T \operatorname{diag}(X^\T X\1)v-v^\T X^\T Xv\ 
&=\sum_{i=1}^n \langle x_i,n\mu\rangle v_i^2-\|Rv\|^2\ \\ 
&\geq \min_i\langle x_i,n\mu\rangle-\|R\|_2^2\ \\
&= n\Big(\|\mu\|^2+\min_i\langle r_i,\mu\rangle\Big)-\|R\|_2^2.
\end{align*}
We will show that the right-hand side is positive with high probability.

First, $\langle \mu,x\rangle\sim N(1,\sigma^2/n)$, and so the Chernoff bound gives
\[
\mathbb P(\|\mu\|<1/2)
\leq\mathbb P(\langle \mu,x\rangle<1/2)
\leq e^{-n/(8\sigma^2)}.
\]
Next, the sample mean $\mu$ and residual matrix $R$ are jointly Gaussian with zero covariance, and therefore independent.
Since $r_i\sim N(0,(1-1/n)\sigma^2I_d)$, it follows that, conditional on $\mu$, each $\langle r_i,\mu\rangle$ has distribution $N(0,(1-1/n)\sigma^2\|\mu\|^2)$.
We apply the union bound, the law of total probability, and the Chernoff bound to obtain
\[
\mathbb{P}\Big(
\min_i\langle r_i,\mu\rangle<-s\|\mu\|
\Big)
\leq n\cdot\mathbb{E}\bigg[\mathbb{P}\Big(
\langle r_1,\mu\rangle<-s\|\mu\|\,\Big|\,\mu
\Big)\bigg]
\leq n\cdot e^{-s^2/(2\sigma^2)}
\]
for every $s\geq0$.
Finally, write $E{\coloneqq}\left[\begin{array}{ccc}z_1&\cdots&z_n\end{array}\right]$, where again $z_i{\coloneqq}x_i-x$. 
Then
\[
\|R\|_{2}
=\|E(I_n-\tfrac{1}{n}\1\1^\T )\|_{2}
\leq \|E\|_{2},
\]
and so the standard Gaussian matrix norm bound gives
\[
\mathbb{P}\Big(
\|R\|_{2}>\sigma\big(\sqrt d+\sqrt n+t\big)
\Big)
\leq 2e^{-t^2/2}
\]
for every $t\geq0$.

Now take $s=1/4$ and $t=\sqrt n$. 
Since $d\le n$, the preceding
estimates imply that
\[
\|\mu\|
\geq \frac{1}{2},
\qquad
\min_i\langle r_i,\mu\rangle
\geq -\frac{1}{4}\|\mu\|,
\qquad
\|R\|_{2}
\leq 3\sigma\sqrt n
\]
hold simultaneously except on an event of probability at most
\[
e^{-n/(8\sigma^2)}
+n e^{-1/(32\sigma^2)}
+2e^{-n/2}
<\epsilon,
\]
where the last step applies our assumption $n\geq C\log(n^2/\epsilon)$ with $C$ sufficiently large and our hypothesis on $\sigma^2$ with $c$ sufficiently small.
Overall, with probability greater than $1-\epsilon$, it holds that every unit vector $v$ orthogonal to $\1$ satisfies
\[
v^\T Lv
\geq n\Big(\|\mu\|^2+\min_i\langle r_i,\mu\rangle\Big)-\|R\|_{2}^2
\geq n\Big(\tfrac{1}{8}-9\sigma^2\Big)
>0,
\]
where the last step holds provided $c$ is sufficiently small.
Thus, $L$ is positive semidefinite, as desired.
\end{proof}

\section{Numerical Results and Applications}
\label{sec:numerics}
In this section, we perform numerical experiments with Fr\'echet means. 
Section~\ref{sec:align_vs_cert} compares different ways to certify the correctness of a Fr\'echet mean,
Section~\ref{sec:acc_runtime} compares the efficacy of the numerical methods from Section~\ref{sec:relaxations}, and the remaining subsections apply these methods to various datasets.

\subsection{Gap Between Alignment and Fr\'echet Optimality}
\label{sec:align_vs_cert}

\begin{figure}[t]
    \centering

    \begin{subfigure}[t]{0.32\textwidth}
        \centering
        \includegraphics[width=\textwidth, trim=0 0 5cm 0, clip]{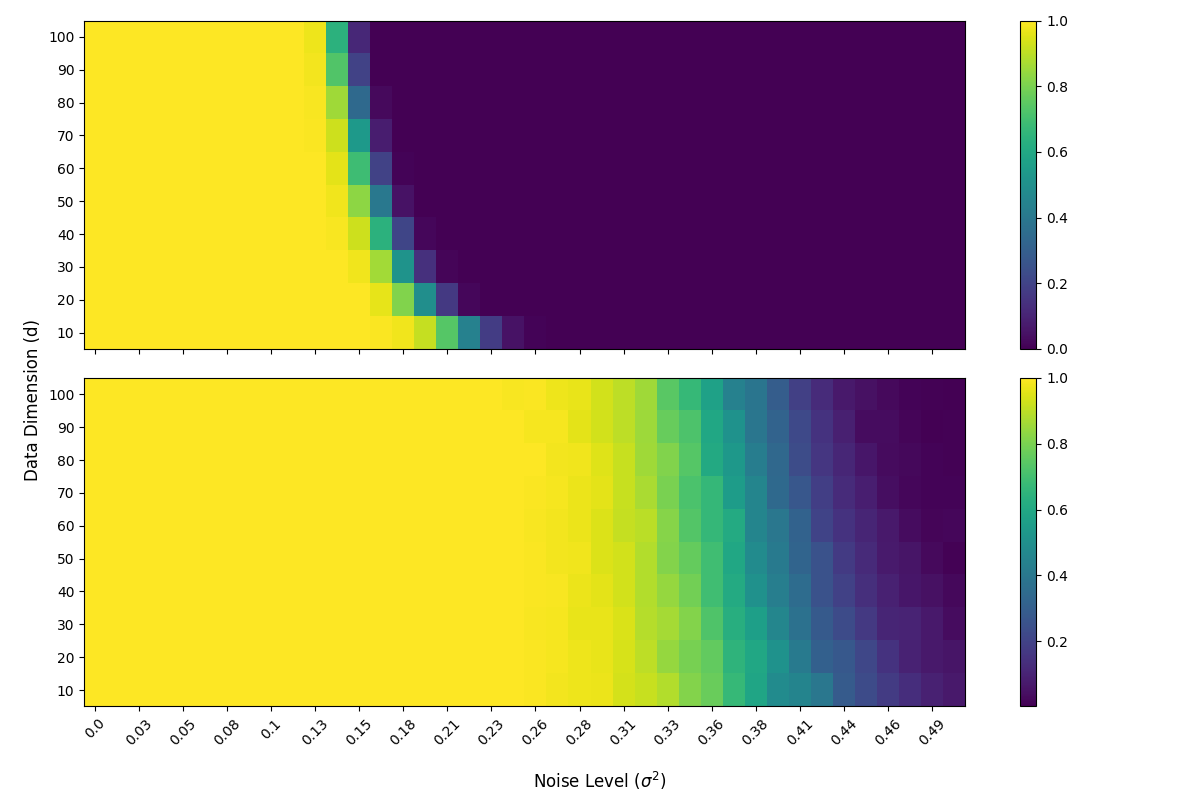}
        \caption{$G = C_2$}
        \label{fig:align_vs_cert:pm}
    \end{subfigure}
    \hfill
    \begin{subfigure}[t]{0.32\textwidth}
        \centering
        \includegraphics[width=\textwidth, trim=0 0 5cm 0, clip]{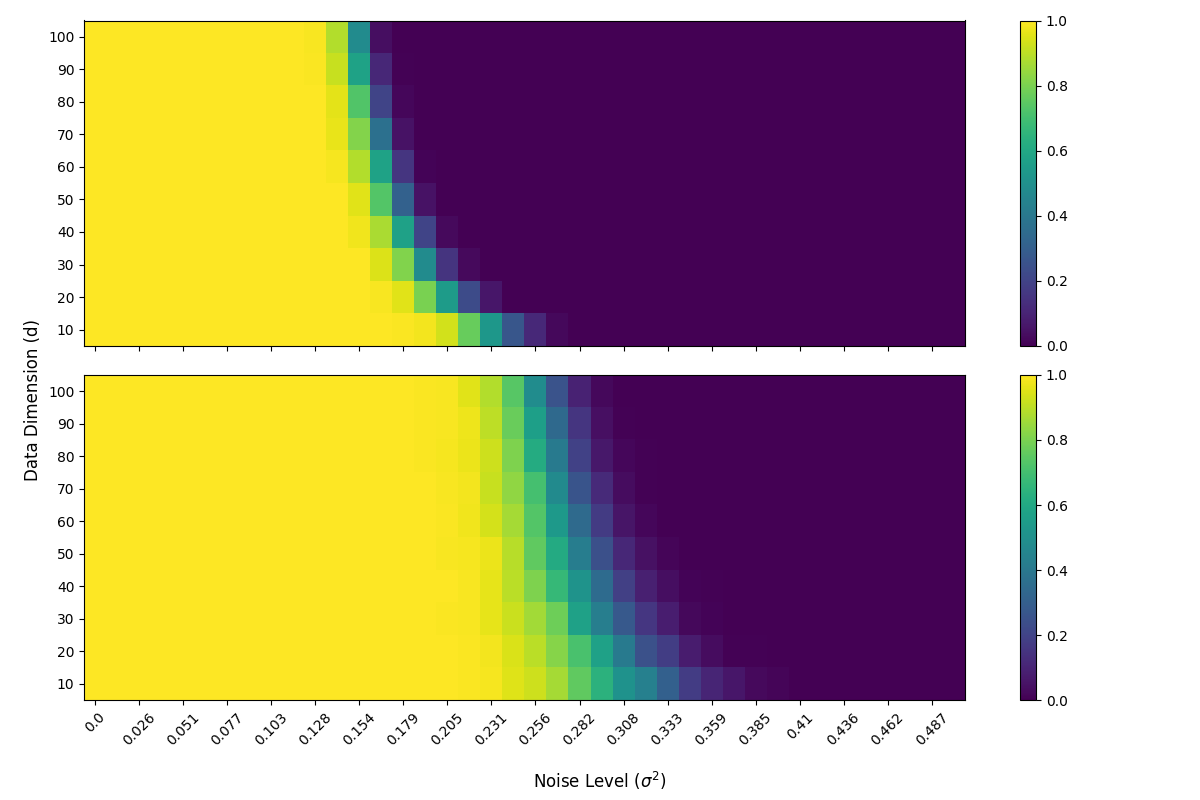}
        \caption{$G = C_3$}
    \end{subfigure}
    \hfill
    \begin{subfigure}[t]{0.32\textwidth}
        \centering
        \includegraphics[width=\textwidth, trim=0 0 5cm 0, clip]{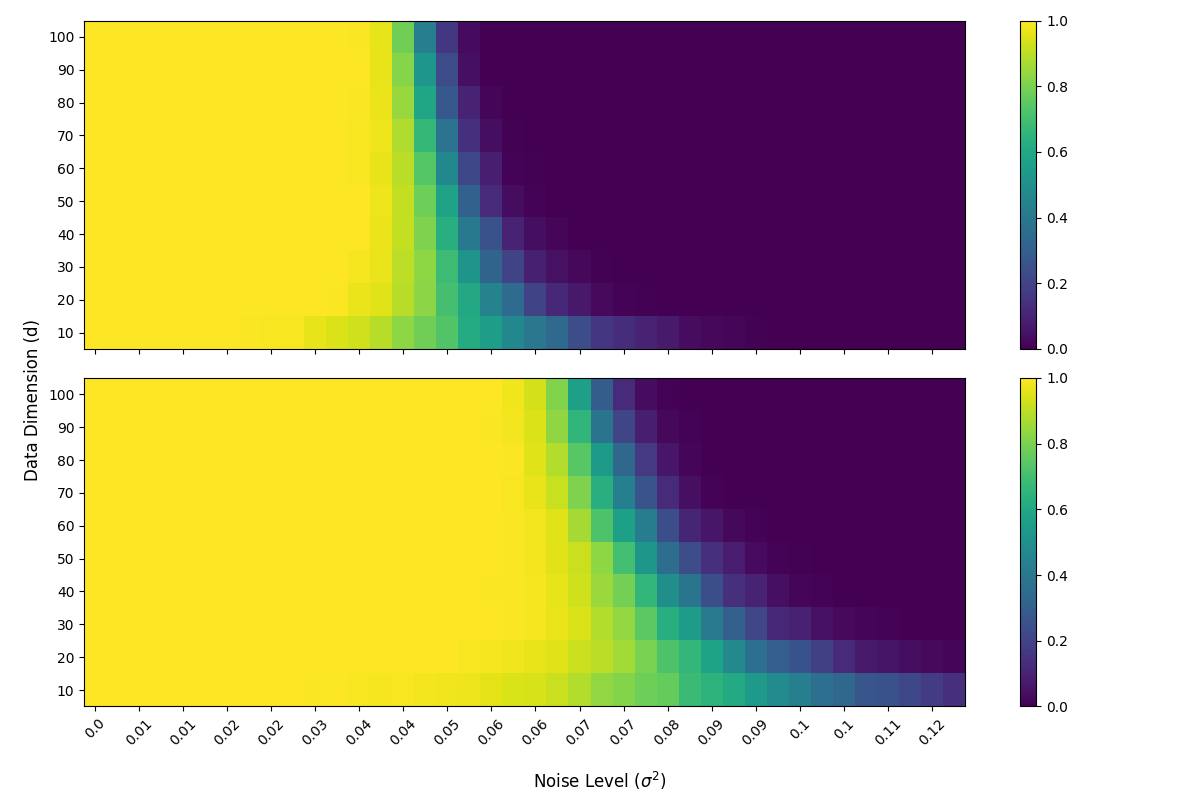}
        \caption{$G = C_d$}
    \end{subfigure}

    \vskip\baselineskip

    \begin{subfigure}[t]{0.32\textwidth}
        \centering
        \includegraphics[width=\textwidth, trim=0 0 5cm 0, clip]{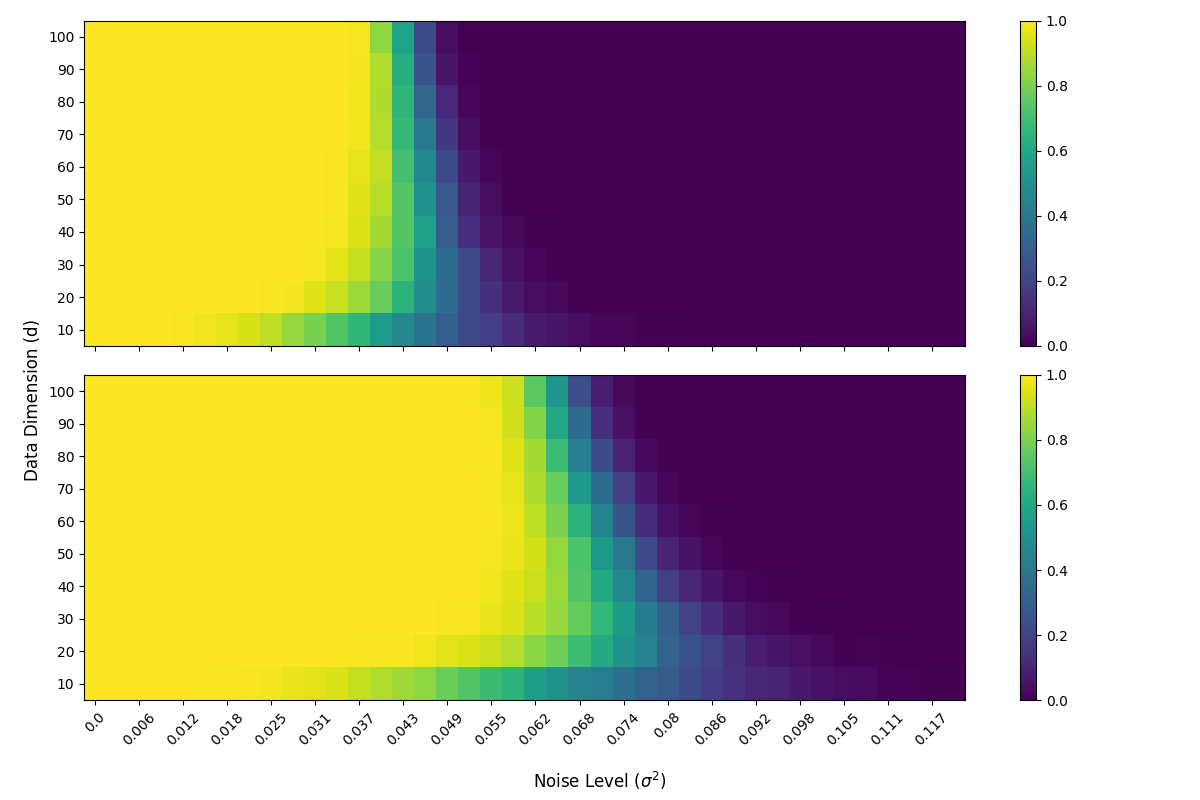}
        \caption{$G = D_d$}
    \end{subfigure}
    \hfill
    \begin{subfigure}[t]{0.32\textwidth}
        \centering
        \includegraphics[width=\textwidth, trim=0 0 5cm 0, clip]{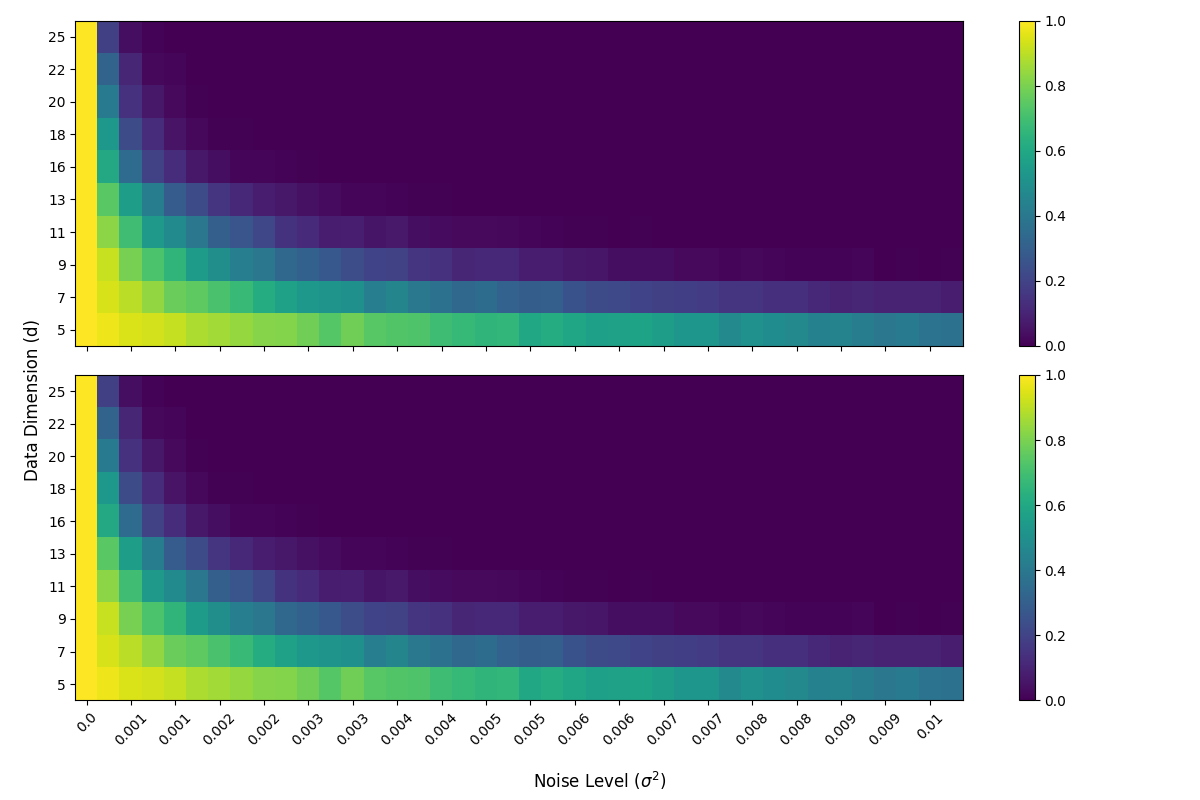}
        \caption{$G = S_d$}
        \label{fig:align_vs_cert:perm}
    \end{subfigure}
    \hfill
    \begin{subfigure}[t]{0.32\textwidth}
        \centering
        \includegraphics[width=\textwidth, trim=0 0 5cm 0, clip]{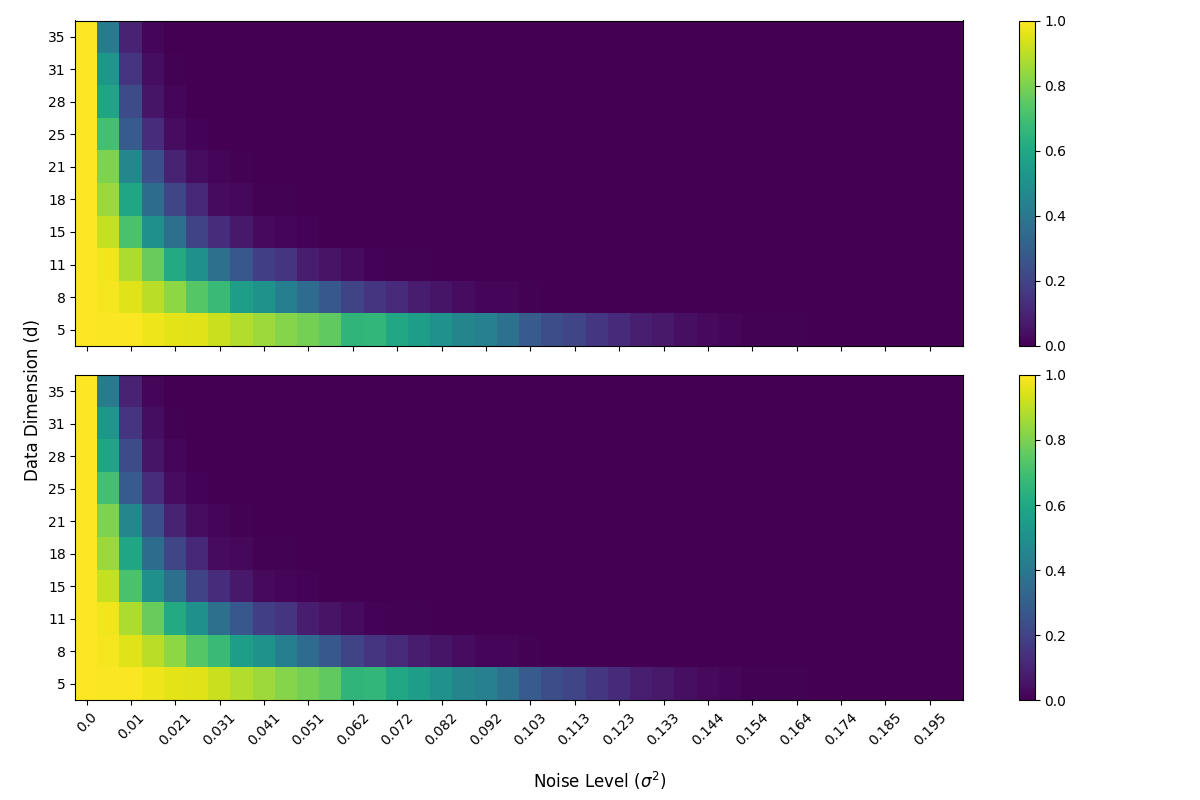}
        \caption{$G = C_2^{d-1}$}
        \label{fig:align_vs_cert:signs}
    \end{subfigure}

    \caption{Empirical probability of alignment (top) versus empirical probability that the Laplacian condition of Corollary~\ref{cor:alignment_L} holds (bottom) for several group actions on Gaussian data. Color indicates probability from purple (near~$0$) to yellow (near~$1$). The representations are (a)~$\pm I_d$, (b)~rotations by multiples of $2\pi/3$ tensored with $I_{d/2}$, (c)~cyclic permutations, (d)~cyclic permutations with reversal, (e)~coordinate permutations, and (f)~even sign flips.}
    \label{fig:align_vs_cert}
\end{figure}

As discussed in Section~\ref{sec:alignment}, aligned data points are optimal orbit representatives for the Fr\'echet mean. 
The Laplacian condition in Corollary~\ref{cor:alignment_L} relaxes this notion of alignment, providing a sufficient condition that is more readily satisfied for some group actions.
While Proposition~\ref{prop.bounds in figure of merit} quantified the difference between these conditions in the special case where $G=\pm$, we now examine other cases by numerical means in Figure~\ref{fig:align_vs_cert}.

We test the probability that random data satisfies either sufficient condition across various noise levels and ambient dimensions.
For each combination of parameters, $n=100$ vectors are drawn independently from a Gaussian distribution centered at a uniformly random point on $S^{d-1}$. 
The horizontal axis captures the variance, while the vertical axis captures the ambient dimension $d$. 
We record the proportion of the $1000$~trials for which the generated data is aligned or satisfies the Laplacian condition.

Figure~\ref{fig:align_vs_cert:pm} illustrates behavior in the case of $G=\pm$ described by  Proposition~\ref{prop.bounds in figure of merit}.
As the ambient dimension increases, smaller noise levels are required for alignment. 
In contrast, the probability that the data satisfies the Laplacian condition is relatively stable. 
Other groups show similar behavior. 
In settings where the ambient dimension is comparable to the size of the group, the Laplacian condition shows a clear improvement over alignment.
These results demonstrate settings in which orbit representatives need not be pairwise aligned to be optimal for the Fr\'echet mean.

Figure~\ref{fig:align_vs_cert:perm} shows that the two sufficient conditions are satisfied with equal probability across all tests. 
This is consistent with our discussion of alignment. 
In the case $G=S_d$, sorting all vectors in the same order simultaneously minimizes every pairwise distance (see Example~\ref{ex:perm_alignment}). 
Consequently, every dataset is alignable, and Fr\'echet-optimal representatives must therefore be aligned by Theorem~\ref{thm:euc_mean}.
While visually similar, this is not the case in Figure~\ref{fig:align_vs_cert:signs}. 
There is in fact a small difference in computed probabilities between the two conditions.
That said, the size of the group in this case ($2^{d-1}$) makes it highly likely that a different set of orbit representatives is Fr\'echet optimal whenever alignment fails.

\subsection{Comparison of Methods}
\label{sec:acc_runtime}

\begin{figure}[t]
    \centering
    \begin{subfigure}{0.48\textwidth}        \includegraphics[width=\linewidth]{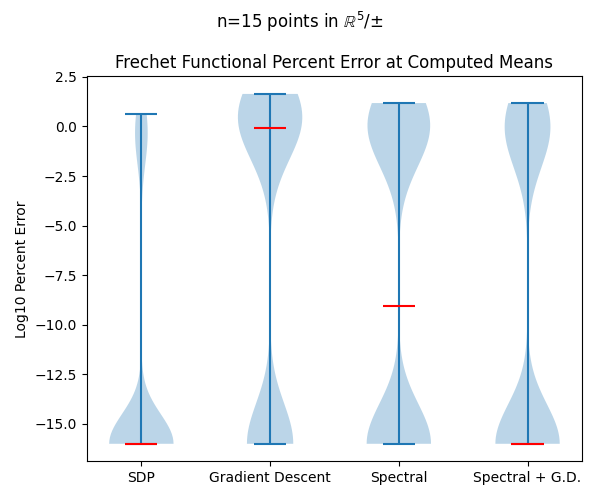}
    \end{subfigure}
    ~
    \begin{subfigure}{0.48\textwidth}
        \includegraphics[width=\linewidth]{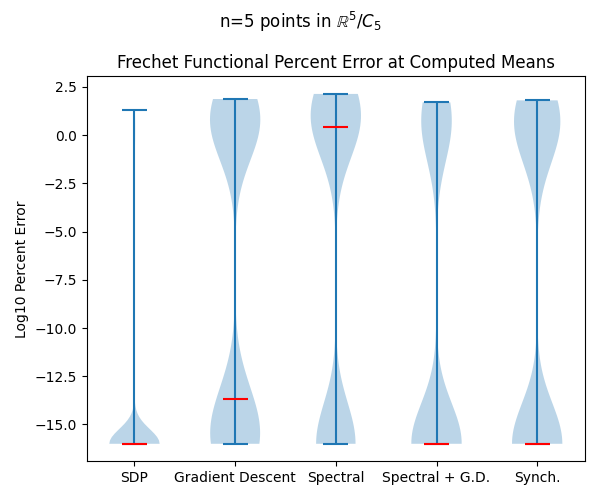}
    \end{subfigure}
    \caption{Distribution of percent error in the Fr\'echet functional at the computed means across $1000$ trials. Median errors are shown in red.}
    \label{fig:error_plots}
\end{figure}

\begin{figure}[t]
    \centering
    \begin{subfigure}{0.48\textwidth}        
    \includegraphics[width=\linewidth]{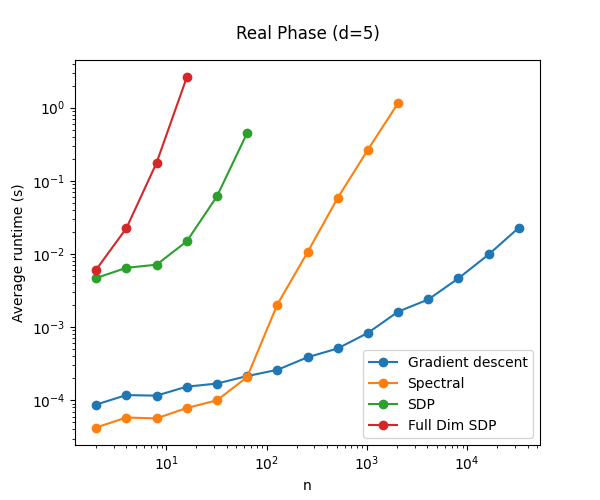}
    \end{subfigure}
    ~
    \begin{subfigure}{0.48\textwidth}
        \includegraphics[width=\linewidth]{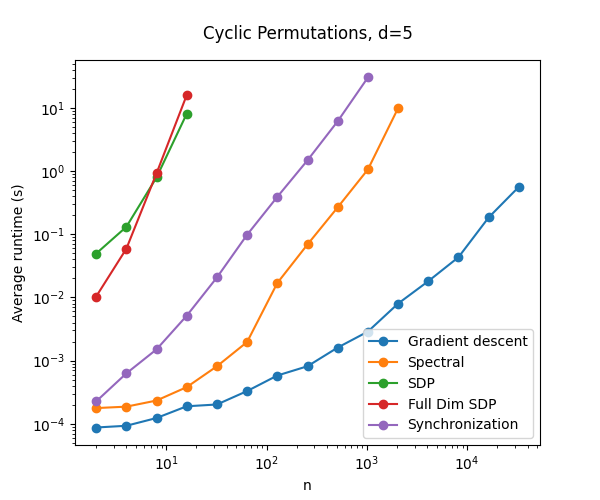}
    \end{subfigure}
    \caption{Runtime to compute an approximate Fr\'echet mean.}
    \label{fig:time_plots}
\end{figure}

Next, we compare the accuracy and runtime of different methods that compute approximate Fr\'echet means.
In order to determine the ground-truth Fr\'echet mean, we choose a small group action acting on a small $n$-point dataset.
In this setting, a brute-force $\mathcal{O}(|G|^{n})$ combinatorial approach is feasible to use as a baseline.

We measure error in the computed Fr\'echet mean using percent error in the Fr\'echet functional~\eqref{eqn:frechet}.
Figure~\ref{fig:error_plots} shows the log-distribution of errors for each method in the settings of $\mathbb{R}^5/\pm$ and $\mathbb{R}^5/C_5$ acting by cyclic permutations. 
Each distribution is computed from $1000$ independently sampled Gaussian datasets with unit variance and a center chosen uniformly at random from the sphere.
For the same settings,
Figure~\ref{fig:time_plots} displays the average runtime of each method over $100$ trials for increasing dataset sizes. 
In the case of $G=\pm$, the synchronization approximation is equivalent to the spectral method and is therefore omitted. 

In both settings, the rounded SDP relaxation performs remarkably well in computing the Fr\'echet mean with median error at machine precision.
The price of this accuracy is runtime, as both SDP methods take significantly longer than the others.
Using the reduced-form SDPs in Sections~\ref{sec:pm_sdp} and~\ref{subsub.sdp cyclic} reduces overall runtime, but they scale similarly with dataset size.

For gradient descent, we used the alternating method for finding Karcher means discussed in Section~\ref{sec:karcher_are_euc}, as it is equivalent to gradient descent with a specific step size. 
We chose to initialize at a uniformly random data point for this comparison.
The error distribution indicates that gradient descent frequently finds a local Karcher mean rather than the true Fr\'echet mean. 
In contrast to an SDP approach, gradient descent is computationally efficient and scales well with large datasets. 

The specific spectral methods derived for these actions demonstrate a useful middle ground in both runtime and accuracy. 
The spectral methods provide an initialization-independent approximation, but often have nonzero errors since they do not necessarily produce even a Karcher mean. 
Using the spectral method to provide an initialization for gradient descent performs better than either the spectral method or gradient descent alone, and scales as well as the spectral method itself. 
This suggests that when $n$, $d$, or $|G|$ is too large for an SDP, performing a spectral approximation followed by gradient steps provides a robust and tractable way to compute an approximate Fr\'echet mean.

\subsection{MNIST Digits}
\label{sub.mnist}
\begin{figure}[ht]
    \centering
    \begin{subfigure}{0.8\linewidth}
        \centering
        \includegraphics[width=\linewidth]{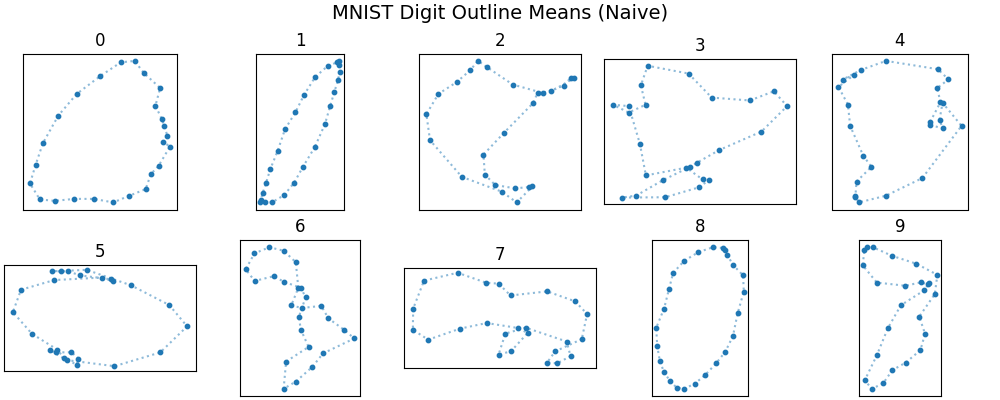}
        \caption{Euclidean means}
    \end{subfigure}
    \hfill
    \begin{subfigure}{0.8\linewidth}
        \centering
        \includegraphics[width=\linewidth]{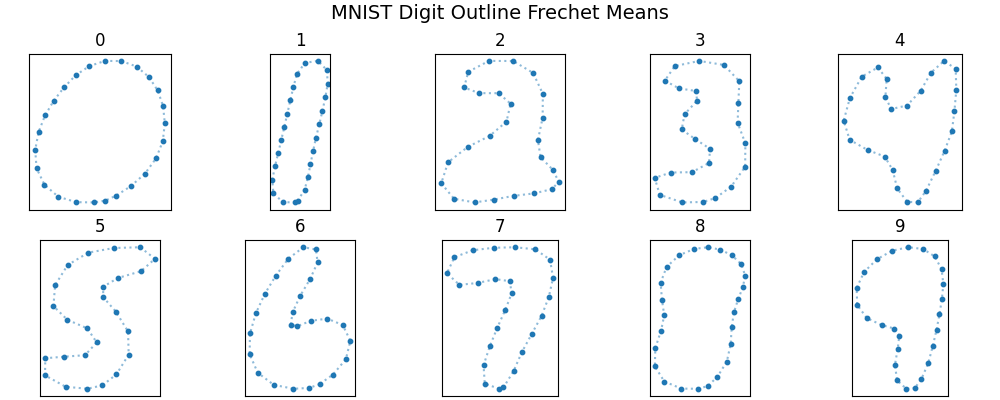}
        \caption{Fr\'echet means in $\mathbb{R}^{25 \times 2}/{C_{25}}$}
    \end{subfigure}
    \caption{Euclidean and approximate Fr\'echet means of MNIST digit outlines.}
    \label{fig:MNIST}
\end{figure}

In this section, we consider a modified version of the MNIST database of handwritten digits. 
Each image is processed to extract the digit boundaries, and $25$ points are sampled uniformly along each boundary in a loop.  
The starting point of this loop is chosen arbitrarily, introducing a cyclic symmetry.
As such, we think of this MNIST outline data as residing in $\mathbb{R}^{25 \times 2} / C_{25}$, where $C_{25}$ acts by cyclically permuting the points along the loop.

We partition the digits according to their digit label and compute the Fr\'echet mean of each class using the spectral method detailed in Section~\ref{sec:cyclic_spectral_algo}. 
The resulting in-class Fr\'echet means are shown in Figure~\ref{fig:MNIST}, which also displays the naive Euclidean means for comparison.
By working directly in the quotient space, we avoid an explicit alignment preprocessing step, which might fail to align the data uniformly.  
The computed Fr\'echet means reflect the true structure of the data, independent of the ambiguity in the starting point.

\subsection{Districting Plans}\label{subsec:districting_plans}

\begin{figure}
    \centering
    
    \begin{subfigure}{0.35\linewidth}
        \centering
        \includegraphics[width=\linewidth]{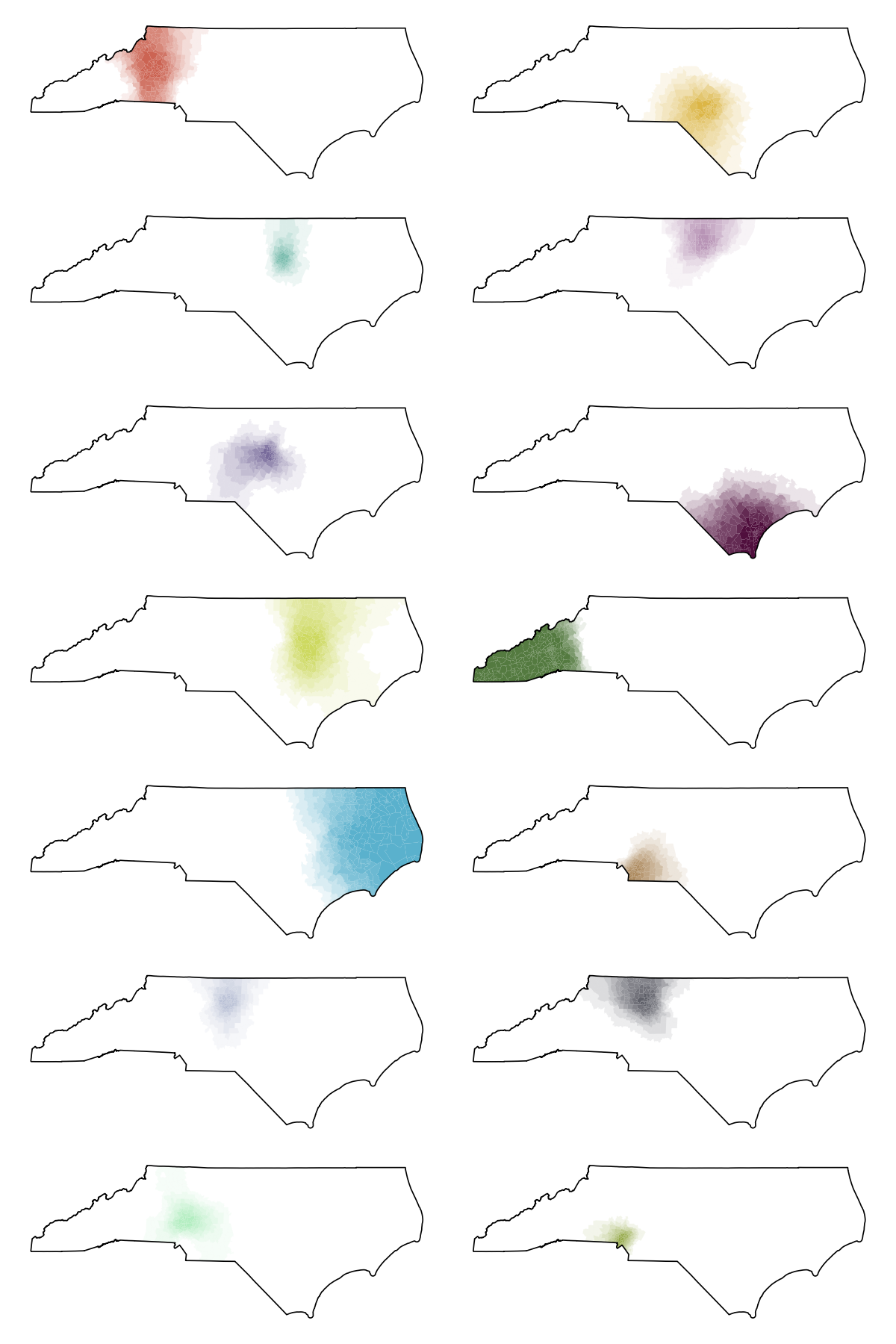}
        \caption{Individual districts}
    \end{subfigure}
    \hfill
    \begin{subfigure}{0.63\linewidth}
        \centering

        \begin{subfigure}{\linewidth}
            \centering
            \includegraphics[width=\linewidth, trim=0 160 0 180, clip]{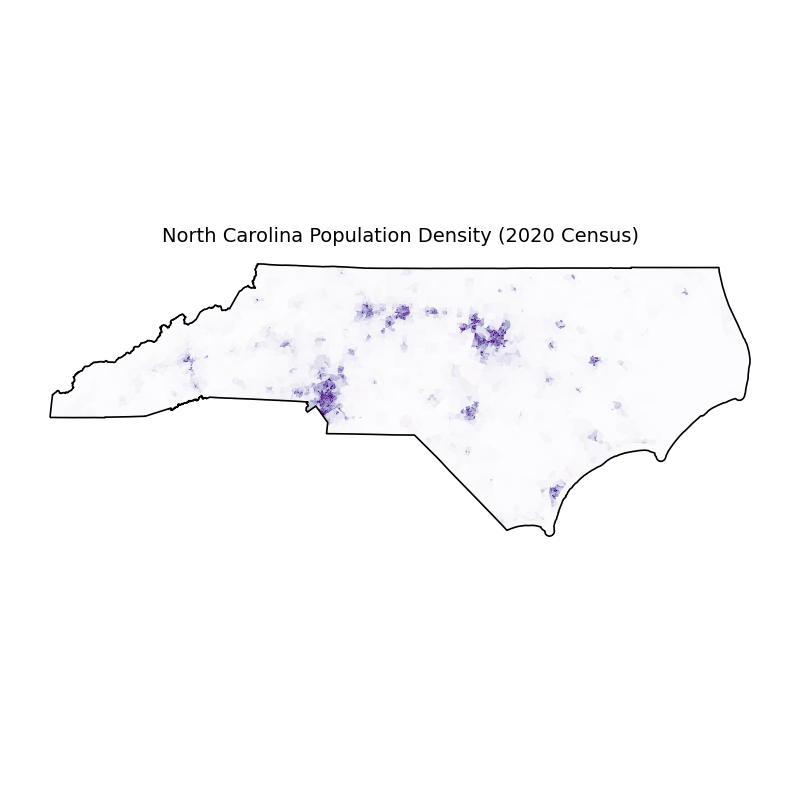}
            \caption{Population density by precinct (2020 Census)}
        \end{subfigure}

        \begin{subfigure}{\linewidth}
            \centering
            \includegraphics[width=\linewidth, trim=0 0 0 25, clip]{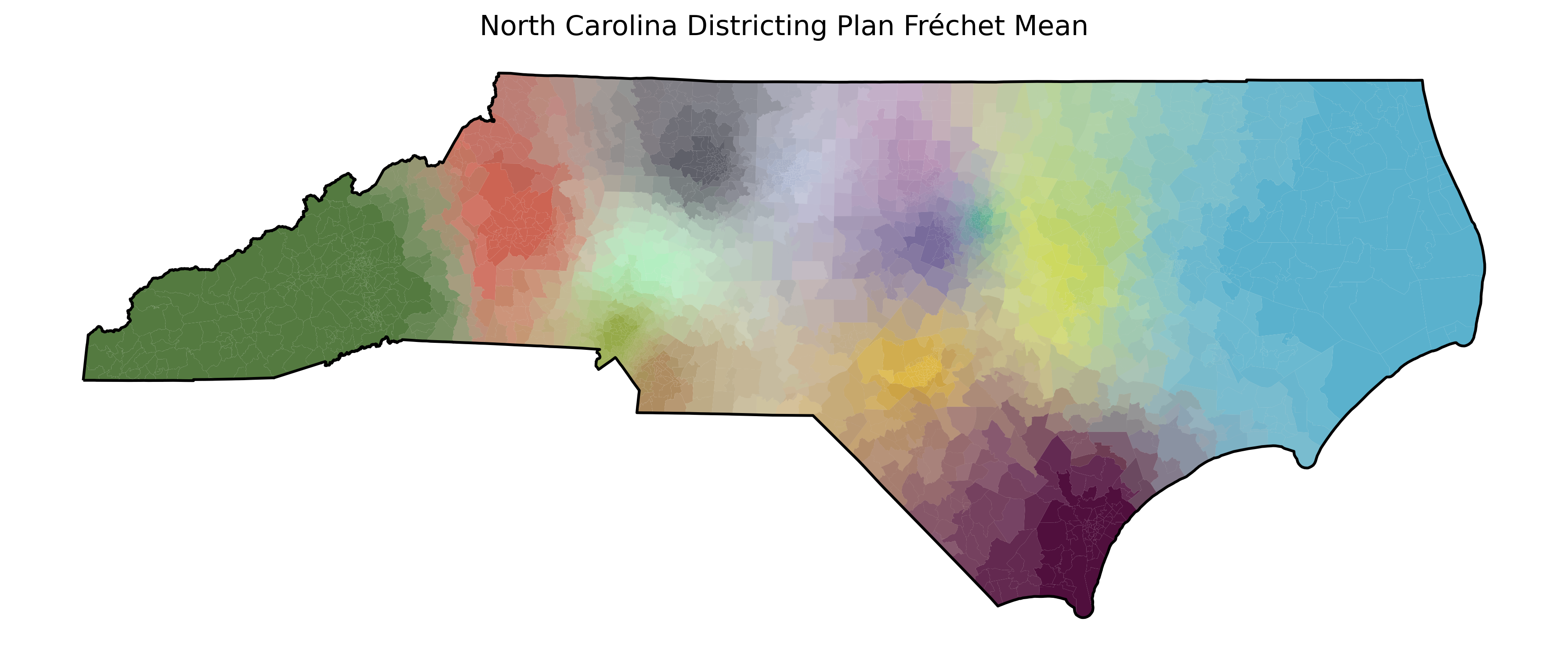}
            \caption{Approximate Fr\'echet mean of $1000$ MCMC plans}
        \end{subfigure}
        
    \end{subfigure}
    
    \caption{Approximate Fr\'echet mean of districting plans in North Carolina.}
    \label{fig:districts}
\end{figure}

In this section, we compute an approximate Fr\'echet mean of congressional districting plans that exhibit an inherent permutation symmetry.
We analyze $1000$ plans generated using the \emph{cycle walk algorithm}, an MCMC-based scheme for sampling graph partitions that was specifically developed to create political districting plans~\cite{deford2025cycle}.
Each plan is represented as a $14 \times 2650$ binary matrix, where entry $(i,j)$ equals $1$
when the $j$th voting precinct is assigned to the $i$th district.
Accordingly, each column contains exactly one nonzero entry.
Since district labels are not canonical, we view this data as residing in the orbit space $\mathbb{R}^{14 \times 2650} / S_{14}$.

To compute an approximate Fr\'echet mean, we apply the spectral approach described in Algorithm~\ref{alg:spectral_general}, followed by gradient descent.
Notably, this data does not require rounding to a corresponding synchronization problem.
The resulting Fr\'echet mean
can be interpreted as a ``fuzzy assignment'' of precincts to districts.
See Figure~\ref{fig:districts} for a visualization. 
For clarity, we separately display each fuzzy district as well as the population. 
A similar experiment was performed in~\cite[Section 5.1]{needhamGeometricAveragesPartitioned2021}, using a different dataset and a different Fr\'echet mean approximation algorithm. In particular, the dataset in that paper used pre-2020 census data, for which North Carolina was assigned 13 districts (rather than the 14 district plans used here). Despite these differences, the majority of the fuzzy districts we obtain here have clear geographic counterparts in the result presented in the previous work.

\subsection{Clustering with Symmetries}
\label{sec:applications}

In the following sections, we implement $k$-means clustering in the context of orbit spaces.
We define $k$-Fr\'echet mean clustering via the optimization problem
\begin{equation}
    \label{eqn:kfrechet}
    \argmin_{\substack{C_1,\dots,C_k \\ C_1 \sqcup \dots \sqcup C_k = \{ 1, \dots, n\}}}
    \sum_{i=1}^k \sum_{j \in C_i} d([\mu_i],[x_j])^2,
\end{equation}
where $\mu_i$ denotes the Fr\'echet mean of the cluster $\{[x_j]\}_{j\in C_i}$ with respect to the quotient metric.

This formulation seeks a partition that minimizes the total within-cluster variance, as measured by the Fr\'echet functional.
This provides a natural way to perform unsupervised clustering of data in an orbit space.
The standard iterative $k$-means algorithm can be readily adapted to this setting by alternating between assigning clusters and computing Fr\'echet means.
Orbits are assigned to clusters according to the nearest cluster mean in the quotient metric, and then cluster means are updated using an appropriate method from Section~\ref{sec:relaxations}. 

\subsubsection{Cylinder, Bell, Funnel}

The Cylinder--Bell--Funnel dataset~\cite{timeseriesclassificationTimeSeries} was introduced to compare classification methods for time series data.
It consists of realizations of three types of time series: cylinders, which resemble indicator functions over a central interval; bells, which exhibit a gradual increase followed by a step down; and funnels, which display a step up followed by a gradual decrease.
Each realization is corrupted by noise in both its amplitude and its time domain, leading to perturbations in signal height and horizontal alignment.
More details can be found in Example~4.7 of~\cite{saito1994local}. 
The data is split into a training set of only $30$ time series of length $128$ and $900$ test samples. 
In this section, we cluster these time series. 

We first apply $k$-means clustering to the training set with $k=3$.
The results of standard $k$-means are shown in the top-left panel of Figure~\ref{fig:CBF}.
Using the resulting cluster centers to classify the test set yields an accuracy of only $65\%$.
Since the data exhibit variability in temporal alignment, the orbit space $\mathbb{R}^{128}/C_{128}$ provides a more natural setting for clustering.
We therefore apply the $k$-Fr\'echet means algorithm to perform clustering in the quotient space, using the spectral method of Section~\ref{sec:cyclic_spectral_algo} followed by gradient descent to compute cluster means in each iteration.
The resulting centroids (defined up to cyclic shift) are shown in the bottom-left panel of Figure~\ref{fig:CBF}.
Visually, this approach recovers the underlying shapes more accurately than standard $k$-means.
In fact, classification using these centroids achieves a $90\%$ accuracy on the test set.

We next focus our attention on the clustering behavior of sliding windows.
We concatenate the Cylinder, Bell, Funnel training set into a single time series of length $3840$ and extract all $3713$ overlapping windows of length $128$.
This setup follows the experiments in Section 4 of~\cite{keogh2005clustering} and Section 3 of~\cite{alexeev2025clustering}.
In both works, applying $k$-means to the sliding window representation is found to be completely meaningless, as the resulting centroids resemble sinusoids (see the top-right panel of Figure~\ref{fig:CBF}).
The authors of~\cite{alexeev2025clustering} attribute this ambiguity to an approximate cyclic symmetry in the sliding window construction: consecutive windows differ by a cyclic permutation and the replacement of one boundary sample.
We instead perform clustering in the orbit space and report the resulting centroids in the bottom-right panel of Figure~\ref{fig:CBF}.
By accounting for this symmetry, the resulting clusters become meaningful.
Using these centroids to classify the test set yields an accuracy of $81\%$, despite the substantial noise introduced by the sliding window representation.

\begin{figure}
    \centering
    \includegraphics[width=\linewidth]{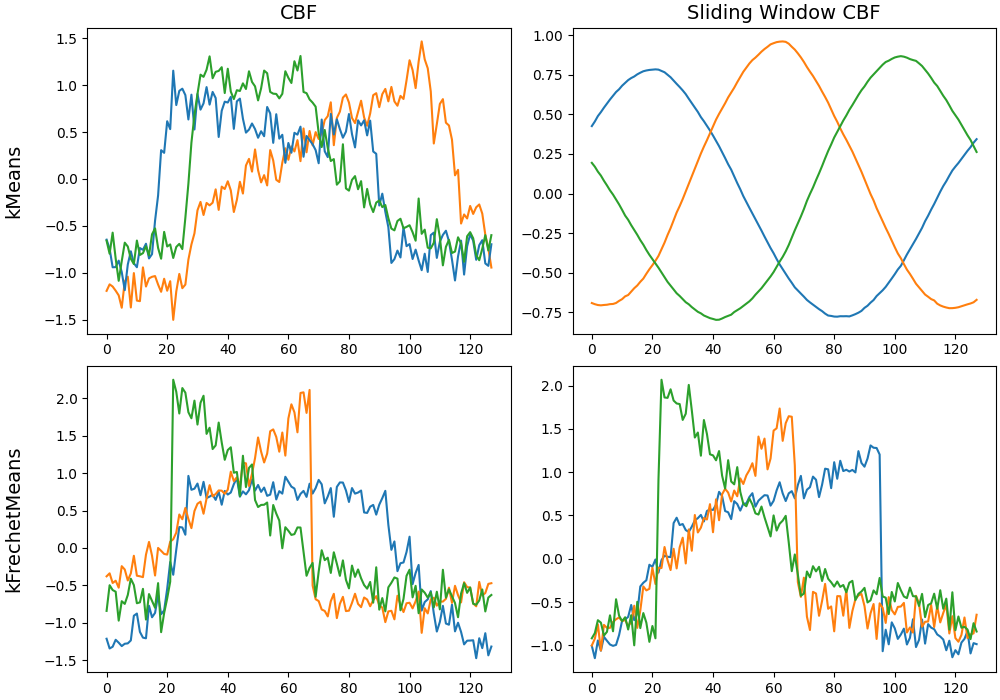}
    \caption{$k$-means centroids from the CBF training set (left) and the sliding-window augmented CBF training set (right). The top and bottom rows show results obtained by working in Euclidean space $\mathbb{R}^{128}$ and the quotient space $\mathbb{R}^{128}/C_{128}$, respectively.}
    \label{fig:CBF}
\end{figure}

\subsubsection{Electrocardiography}

The PTB-XL dataset~\cite{PhysioNet-ptb-xl-1.0.3} consists of $21{,}799$ clinical $12$-lead ECGs from $18{,}869$ patients, each of length 10 seconds sampled at $100$Hz.
Each ECG is annotated by up to two cardiologists, with potentially multiple labels per recording.
We restrict attention to recordings labeled either as myocardial infarction (heart attack) or as healthy ECGs.

We apply several preprocessing steps to the ECG recordings.
First, each recording is truncated to a one-second segment.
We then use a simple peak-detection procedure to ensure that each segment contains a single cardiac peak, padding the remainder with the mean waveform value.
Finally, we normalize each channel to a common scale and center each waveform.
For computational tractability, we subsample $500$ ECGs per class.
The one-second ECG segments each contain a single heartbeat, but are not naturally aligned in time.
To account for this misalignment, we model the data in the orbit space induced by cyclic time shifts, namely, $\mathbb{R}^{100 \times 12}/C_{100}$.

Our first experiment is computing the mean ECG for both classes.
Two leads of the resulting means are shown in the left panel of Figure~\ref{fig:ECG_b}.
The averaged waveforms reveal clear visual differences between healthy subjects and those with myocardial infarction.
In particular, healthy ECGs exhibit a more defined positive bump in lead~$1$ following the main spike and lead 6 shows a distinct downward deflection at the main spike in the infarction class.
We further use these class means to construct a highly interpretable distance-based classifier, labeling a new ECG according to its relative distance to the two Fr\'echet means in the orbit space.
The corresponding ROC curve is shown in Figure~\ref{fig:ROC_a}.
Despite its simplicity,
this classifier achieves an AUC of $0.82$ when evaluated on the held-out test data.

We next perform a separate $k$-Fr\'echet means clustering on the data with $k=2$. 
The resulting centroids are shown on the right of Figure~\ref{fig:ECG_b}. 
Interestingly, these cluster means appear to capture meaningful patterns associated with the underlying classes. 
For example, the post-spike bump in lead~$1$ is clearly present in the orange centroid, while the large negative spike in lead $6$ is reflected in the blue centroid.
In contrast, applying classical $k$-means clustering directly to the ECG data results in a singleton cluster and therefore meaningless centroids. 

\begin{figure}
    \centering
    \begin{subfigure}{0.36\linewidth}
        \centering
        \includegraphics[width=\linewidth]{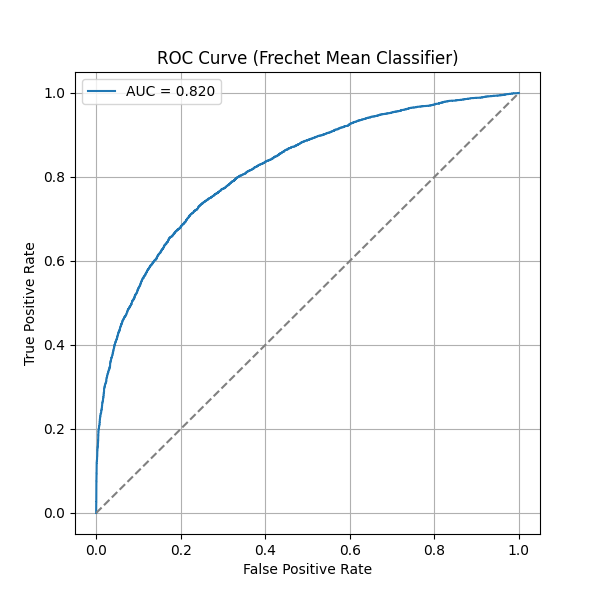}
        \caption{}
        \label{fig:ROC_a}
    \end{subfigure}
    \hfill
    \begin{subfigure}{0.63\linewidth}
        \centering
        \includegraphics[width=\linewidth]{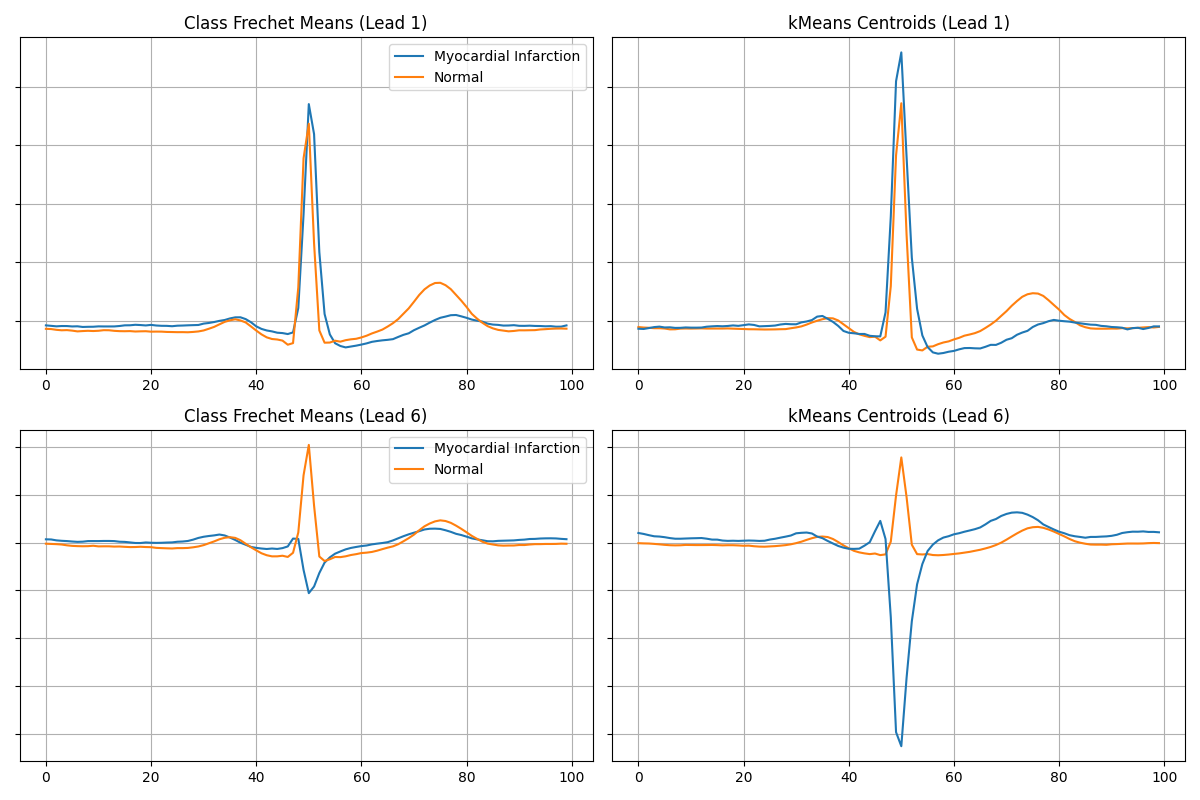}
        \caption{}
        \label{fig:ECG_b}
    \end{subfigure}

    \caption{(a) ROC curve for classifying myocardial infarction based on distance from the Fr\'echet mean of each class. (b) Fr\'echet mean of $500$ randomly sampled heart attack and normal ECG signals alongside the $k$-Fr\'echet-means centroids on the sampled data. The top and bottom panels show two different ECG lead readings.}
    \label{fig:ROC_ECG}
\end{figure}
\section{Discussion}
\label{sec:discussion}

Fr\'echet means in orbit spaces provide a natural notion of averaging in settings with underlying symmetries. 
While computing these means exactly is generally intractable, we introduced a suite of global relaxations that make the problem computationally feasible. 
These relaxations yield practical and robust approaches to approximating Fr\'echet means. 
Moreover, we have established conditions under which a computed solution can be certified as a true Fr\'echet mean using the geometry of the orbit space and semidefinite programming.

This work leaves several directions for future research. 
First, a more principled analysis of the gradient-based approach could give further conditions under which the Fr\'echet functional can be globally minimized quickly. 
Second, many of the numerical methods developed here extend naturally to settings with continuous group actions, such as orthogonal actions on point clouds.
It would be interesting to investigate alignment in these contexts, along with certification via semidefinite programming. 
Finally, more work is needed to better understand how orbit space geometry influences the complexity of finding the Fr\'echet mean.

\section*{Acknowledgments}
\label{sec:acknowledgements}

We would like to thank Thomas Weighill for providing the districting data and the code for the visualization used in Section \ref{subsec:districting_plans}.
DGM was supported in part by NSF grant DMS 2220304 and AFOSR grant FA9550-26-1-B224.
TN was partially supported by NSF grant CIF 2526630.
DGM thanks OpenAI for complimentary access to ChatGPT-6 Pro. 
The authors used ChatGPT to find relevant literature, to help generate the figures, and to give feedback on preliminary versions of this manuscript. 
All mathematical claims, citations, and computations were checked by the authors, who take full responsibility for the content of this paper.

\bibliographystyle{plain} 
\bibliography{frechet}

\appendix

\section{Fr\'echet SDP Dual Program}
\label{sec:dual}

This section is dedicated to the dual problem of the Fr\'echet SDP
\eqref{eqn:fullSDP}. We first detail the derivation of the form used in
Lemma~\ref{lem.dual}, then explicitly demonstrate weak and strong duality.

\subsection{Deriving the Dual}

For convenience, we restate the primal problem:
\[
\begin{aligned}
    \max_{Y\in\mathbb{R}^{nd\times nd}}
    \quad & \left\langle xx^\T ,Y\right\rangle\\
    \text{s.t.}\quad
    &Y\succeq 0,\\
    &Y_{[ij]}\in \mathrm{Conv}(G),
        \qquad 1\leq i<j\leq n,\\
    &Y_{[ii]}=I_d,
        \qquad i=1,\ldots,n.
\end{aligned}
\]

Let \(\lambda_{\ell}^{ij}\) be weights indexed by
\(1\leq i<j\leq n\) and \(1\leq\ell\leq |G|\), let \(E^{ij}\) be the
\(n\times n\) elementary matrix with a one in the \((i,j)\)th entry,
and enumerate \(G=\{g_1,\ldots,g_{|G|}\}\). Define
\[
g_{\ell}^{ij}
{\coloneqq}
E^{ij}\otimes g_\ell+E^{ji}\otimes g_\ell^\T .
\]
Then we can enforce the convex hull constraint explicitly by letting
\[
Y_{[ij]}
=
\sum_{\ell=1}^{|G|}
\lambda_{\ell}^{ij}g_\ell
\]
for each \(1\leq i<j\leq n\). Our SDP obtained by these relaxations is
then
\[
\begin{aligned}
    \max_{Y\in\mathbb{R}^{nd\times nd}}
    \quad &\left\langle xx^\T ,Y\right\rangle\\
    \text{s.t.}\quad
    &Y
    =
    I_{nd}
    +
    \sum_{1\leq i<j\leq n}
    \sum_{\ell=1}^{|G|}
    \lambda_{\ell}^{ij}g_{\ell}^{ij},\\
    &Y\succeq 0,\\
    &\lambda_{\ell}^{ij}\geq 0,
        \qquad
        1\leq i<j\leq n,\quad
        1\leq\ell\leq |G|,\\
    &\sum_{\ell=1}^{|G|}
    \lambda_{\ell}^{ij}=1,
        \qquad
        1\leq i<j\leq n.
\end{aligned}
\]

We can now rewrite the objective in terms of the
\(\lambda_{\ell}^{ij}\):
\begin{align*}
\left\langle xx^\T ,Y\right\rangle
&=
\left\langle
xx^\T ,
I_{nd}
+
\sum_{1\leq i<j\leq n}
\sum_{\ell=1}^{|G|}
\lambda_{\ell}^{ij}g_{\ell}^{ij}
\right\rangle\\
&=
\sum_{i=1}^n\|x_i\|^2
+
\sum_{1\leq i<j\leq n}
\sum_{\ell=1}^{|G|}
\lambda_{\ell}^{ij}
\left\langle xx^\T ,g_{\ell}^{ij}\right\rangle\\
&=
\sum_{i=1}^n\|x_i\|^2
+
2\sum_{1\leq i<j\leq n}
\sum_{\ell=1}^{|G|}
\lambda_{\ell}^{ij}
\left\langle x_i,g_\ell x_j\right\rangle.
\end{align*}
If we define
\[
c_{\ell}^{ij}
{\coloneqq}
\left\langle x_i,g_\ell x_j\right\rangle
\]
and let \(c,\lambda\in
\mathbb{R}^{|G|\binom{n}{2}}\) be indexed by \(i,j,\ell\) as above,
then our primal SDP, after dropping the additive constant
\(\sum_{i=1}^n\|x_i\|^2\) and the constant factor \(2\), becomes
\begin{equation}
\label{eqn:simple_primal}
\begin{aligned}
    \max_{\lambda\geq 0}
    \quad &\langle c,\lambda\rangle\\
    \text{s.t.}\quad
    &\sum_{\ell=1}^{|G|}
    \lambda_{\ell}^{ij}=1,
        \qquad
        1\leq i<j\leq n,\\
    &I_{nd}
    +
    \sum_{1\leq i<j\leq n}
    \sum_{\ell=1}^{|G|}
    \lambda_{\ell}^{ij}g_{\ell}^{ij}
    \succeq 0.
\end{aligned}
\end{equation}
From this rewritten primal, we can now write down the dual program. Letting
\(\beta\in\mathbb{R}^{\binom{n}{2}}\) and
\(V\in\mathbb{R}^{nd\times nd}\), our dual is
\begin{equation}
\label{eqn:original_dual}
\begin{aligned}
    \min_{\beta,V}
    \quad & \beta^\T \mathbf{1}+\operatorname{Tr}(V)\\
    \text{s.t.}\quad
    & \beta^{ij}
    \geq
    c_{\ell}^{ij}
    +
    \left\langle g_{\ell}^{ij},V\right\rangle,
    \qquad
    1\leq i<j\leq n,\quad
    1\leq\ell\leq |G|,\\
    & V\succeq 0.
\end{aligned}
\end{equation}

We next rewrite the dual in a more helpful form. By the sparse
structure of \(g_{\ell}^{ij}\), we can write
\[
c_{\ell}^{ij}
+
\left\langle g_{\ell}^{ij},V\right\rangle
=
\frac12
\left\langle
xx^\T +2V,g_{\ell}^{ij}
\right\rangle.
\]
Also notice that the minimal feasible values of \(\beta^{ij}\) are
attained by
\[
\beta^{ij}
=
\frac12
\max_{\ell\in\{1,\ldots,|G|\}}
\left\langle
xx^\T +2V,g_{\ell}^{ij}
\right\rangle.
\]
We can therefore rewrite the objective as
\begin{align*}
\beta^\T \mathbf{1}+\operatorname{Tr}(V)
&=
\operatorname{Tr}(V)
+
\frac12
\sum_{1\leq i<j\leq n}
\max_{\ell\in\{1,\ldots,|G|\}}
\left\langle
xx^\T +2V,g_{\ell}^{ij}
\right\rangle\\
&=
\frac12
\left(
\operatorname{Tr}(S)
+
\sum_{1\leq i<j\leq n}
\max_{\ell\in\{1,\ldots,|G|\}}
\left\langle
xx^\T +S,g_{\ell}^{ij}
\right\rangle
\right),
\end{align*}
where \(S{\coloneqq}2V\succeq 0\).
Define \(W{\coloneqq}xx^\T +S\), and let
\[
\mathcal{H}
{\coloneqq}
\left\{
H\in\mathbb{S}^{nd}:
\begin{array}{l}
H_{[ii]}=I_d,\quad i=1,\ldots,n,\\
H_{[ij]}\in G,\quad 1\leq i<j\leq n
\end{array}
\right\}.
\]
For a fixed \(W\), the maximum over \(H\in\mathcal{H}\) is attained
by choosing
\[
H_{[ij]}
\in
\argmax_{g\in G}
\left\langle g,W_{[ij]}\right\rangle
\]
for each \(1\leq i<j\leq n\), and setting
\(H_{[ji]}=H_{[ij]}^\T \). It follows that
\begin{align*}
\max_{H\in\mathcal{H}}\langle H,W\rangle
&=
\operatorname{Tr}(W)
+
\sum_{1\leq i<j\leq n}
\max_{\ell\in\{1,\ldots,|G|\}}
\left\langle W,g_{\ell}^{ij}\right\rangle\\
&=
\sum_{i=1}^n\|x_i\|^2
+
\operatorname{Tr}(S)
+
\sum_{1\leq i<j\leq n}
\max_{\ell\in\{1,\ldots,|G|\}}
\left\langle xx^\T +S,g_{\ell}^{ij}\right\rangle.
\end{align*}
Hence,
\[
\beta^\T \mathbf{1}+\operatorname{Tr}(V)
=
\frac12
\left(
\max_{H\in\mathcal{H}}\langle H,W\rangle
-
\sum_{i=1}^n\|x_i\|^2
\right).
\]
Restoring the additive constant and factor \(2\) that were removed in
passing from the original primal objective to \eqref{eqn:simple_primal},
we obtain the equivalent dual formulation
\[
    \min_{W\succeq xx^\T }
    \max_{H\in\mathcal{H}}\,\langle H,W\rangle.
\]

\subsection{Weak and Strong Duality}

Here we demonstrate weak duality between the
Fr\'echet SDP and its dual explicitly. Let \(Y\) be primal feasible
and let \(W\succeq xx^\T \) be dual feasible. Since
\(Y\succeq 0\) and \(W-xx^\T \succeq 0\), we have
\[
\left\langle Y,W-xx^\T \right\rangle\geq 0.
\]
It follows that
\begin{align*}
\left\langle Y,xx^\T \right\rangle
&\leq
\langle Y,W\rangle\\
&=
\sum_{i=1}^n
\left\langle I_d,W_{[ii]}\right\rangle
+
2\sum_{1\leq i<j\leq n}
\left\langle Y_{[ij]},W_{[ij]}\right\rangle\\
&\leq
\sum_{i=1}^n
\left\langle I_d,W_{[ii]}\right\rangle
+
2\sum_{1\leq i<j\leq n}
\max_{g\in G}
\left\langle g,W_{[ij]}\right\rangle\\
&=
\max_{H\in\mathcal{H}}
\langle H,W\rangle.
\end{align*}
The second inequality follows from
\(Y_{[ij]}\in \conv G\), and the last line follows from the definition
of \(\mathcal{H}\).

The SDP exhibits strong duality, following from the dual form of Slater's condition. 
Consider the form of the dual in \eqref{eqn:original_dual}. 
Taking
\(V=I_{nd}\succ 0\) will give an interior feasible point. 
For \(i<j\), the nonzero blocks of
\(g_{\ell}^{ij}\) are off-diagonal, whereas the nonzero blocks of
\(I_{nd}\) are diagonal. Therefore
\[
\left\langle g_{\ell}^{ij},I_{nd}\right\rangle=0.
\]
We may now choose
\[
\beta^{ij}
>
\max_{\ell\in\{1,\ldots,|G|\}}
c_{\ell}^{ij}
\]
for every \(1\leq i<j\leq n\). This gives a strictly feasible dual
point. Since the primal is feasible and has a finite optimal value,
strong duality holds for every finite \(G\).

\end{document}